\documentclass[12pt]{article}
\usepackage[centertags]{amsmath}
\usepackage{amsthm}
\usepackage{amsfonts}
\usepackage{amssymb}
\usepackage{newlfont}
\usepackage{delarray}
\usepackage{subfigure}
\usepackage{graphicx}
\usepackage[numbers,sort&compress]{natbib}

\usepackage{paralist}

\newtheorem{lemma}{Lemma}

\newtheorem{algorithm}{Algorithm}
\theoremstyle{definition}

\def\sp{\mathop{\mathrm{sp}}}

\newcommand{\degree}{\ensuremath{^\circ}}

\title{Transport in time-dependent dynamical systems: Finite-time coherent sets}
\author{Gary Froyland$^1$, Naratip Santitissadeekorn$^1$, and Adam Monahan$^2$ \\
$^1$School of Mathematics and Statistics \\ University of New South
Wales\\ Sydney NSW 2052, Australia\\
$^2$School of Earth and Ocean Sciences\\ University of Victoria\\ Victoria BC, Canada}

\begin{document}
 \maketitle

 \begin{abstract}
We study the transport properties of nonautonomous chaotic dynamical systems over a finite time duration. We are particularly interested in those regions that remain coherent and relatively non-dispersive over finite periods of time, despite the chaotic nature of the system. We develop a novel probabilistic methodology based upon transfer operators that automatically detects maximally coherent sets.  The approach is very simple to implement, requiring only singular vector computations of a matrix of transitions induced by the dynamics.  We illustrate our new methodology on an idealized stratospheric flow and in two and three dimensional analyses of European Centre for Medium Range Weather Forecasting (ECMWF) reanalysis data.
 \end{abstract}

\section{Introduction}
\textbf{Finite-time transport of time-dependent or nonautonomous chaotic dynamical systems has been the subject of intense study over the last decade.  Existing techniques to analyse transport have evolved from classical geometric theory of invariant manifolds, where co-dimension 1 invariant manifolds are impenetrable transport barriers.  In this work we take a very different approach, based on spectral information contained in a finite-time transfer (or Perron-Frobenius) operator.  Our technique automatically identifies regions of state space that are maximally coherent or non-dispersive over a specific time interval, in the presence of an underlying chaotic system.  These regions, called coherent sets, are robust to perturbation and are carried along by the chaotic flow with little transport between the coherent sets and the rest of state space.  Thus, these coherent sets are ordered skeletons of the time-dependent dynamics around which the chaotic dynamics occurs relatively independently over the finite time considered.  We develop the theory behind an optimization problem to determine these coherent sets and describe in detail a numerical implementation.  Numerical results are given for a model system and real-world reanalysed data.}

Transport and mixing properties of dynamical systems have received considerable interest over the last two decades;  see e.g. \cite{aref_02,Meiss92,wiggins92,wiggins05} for discussions of transport phenomena. A variety of dynamical systems techniques have been introduced to explain transport mechanisms, to detect barriers to transport, and to quantify transport rates.
These techniques typically fall into two classes:  geometric methods, which exploit invariant manifolds and related objects as organizing structures, and probabilistic methods, which study the evolution of probability densities. Geometric methods include the study of invariant manifolds, the theory of lobe dynamics \cite{romkedar_etal90,RW90,wiggins92} in two- (and some three-) dimensions, and the notions of finite-time hyperbolic material lines \cite{haller00} and surfaces \cite{haller01}. The latter objects are often studied computationally via finite-time Lyapunov exponent (FTLE) fields \cite{haller00,shadden05}.  All of these geometric objects represent \emph{transport barriers} and in this way influence (mitigate) global transport.
Probabilistic approaches include a study of almost-invariant sets \cite{DJ97,FD03,froyland05,froyland08} and very recently, coherent sets \cite{FLQ08,FLS10}. For autonomous and time-dependent systems respectively, almost-invariant sets and coherent sets represent those regions in phase space which are \emph{minimally dispersed} under the flow.  Such regions provide an ordered skeleton often hidden in complicated flows.
A recent comparison of the geometric and probabilistic approaches is given in \cite{FP09} for the time-independent and time-periodic settings.

The probabilistic methodologies provide important transport information that is often not well resolved by geometric techniques. Minimally dispersive regions need not be identified by geometric approaches.  For example, recent work \cite{FP09} has shown that regions enclosed by FTLE ridges need not represent maximal transport barriers.
Several authors \cite{MM10,haller10} have noted other shortcomings of the FTLE-based approach: potential ambiguity in multiple FTLE ``ridges'', ambiguity over flow duration for FTLE calculations,
and a lack of correspondence between the strength of the ridge and
the dispersal of mass across the ridge.

Probabilistic techniques have also been shown to be valuable analysis tools for geophysical systems. In such systems, physical quantities are often used to determine transport barriers.  For example, lines of constant sea surface height (as proxies for streamlines under the assumption of geostrophy) are
commonly used to determine locations of rotational trapping regions such as anticyclonic eddies and gyres \cite{CBT07}, and maximum
gradients of potential vorticity (PV) are used to determine
``edges'' of vortices in the stratosphere \cite{MP83,Nash_etal96,SPW97}.  In both of these
geophysical settings, the use of physical quantities has been shown to be
non-optimal in determining the location of transport barriers
\cite{FPET07,SFM10}.


Probabilistic and transfer operator approaches \cite{DJ97,DJ99,FD03,SB07,froyland05,froyland08,BS08,FP09} have proven to be very
effective for autonomous systems.
Initial progress has been
made in the development of these techniques for time-dependent systems \cite{FLQ08,FLS10} over \emph{infinite} time horizons.
In the present work, we focus on transport
analysis of time-dependent systems over a \emph{finite} period of time,
significantly expanding on concepts introduced in \cite{SFM10}, and developing finite-time analogues of the time-asymptotic coherent set constructions in \cite{FLQ08,FLS10}.
We develop methodologies to identify those regions of phase space that are minimally dispersive, or maximally coherent, under the flow, for a specific finite time interval.
We demonstrate the efficacy of our approach on two examples:  an idealized stratospheric flow and a flow obtained from assimilated data sourced from the European Centre for Medium Range Weather forecasting.
In the first example, we demonstrate that our new methodology easily detects an important dynamical separation of the domain;  this separation is not clearly evident from an examination of the finite-time Lyapunov exponent (FTLE) field \cite{haller00,haller01,shadden05}.
In the second example, we show that our new techniques can isolate the Antarctic polar vortex to high accuracy on a two-dimensional isentropic surface when compared to the commonly used potential vorticity criterion \cite{MP83}. We also illustrate our technique directly in three dimensions, going beyond the capabilities of existing techniques to image the vortex in three dimensions.

An outline of the paper is as follows.
In section \ref{sect:to} we describe our setting and outline our main computational tool, the transfer operator (or Perron-Frobenius operator), and our numerical approximation approach.
In section \ref{sect:tech} we motivate and detail our new computational approach.
Section \ref{sect:comp} describes the necessary computations and sections \ref{sect:eg1} and \ref{sect:eg2} illustrate our new methodology via two case studies.


%

\section{Flows, coherent sets, and transfer operators}
\label{sect:to}

Let $M\subset \mathbb{R}^d$ be a compact smooth manifold and
consider a time-dependent vector field $f(z,t)$, $z\in M$, $t\in
\mathbb{R}$. Suppose that $f$ is smooth enough for the existence of
a flow map $\Phi(z,t,\tau):M\times \mathbb{R}\times \mathbb{R}\to
M$, which describes the terminal location of an initial point $z$ at
time $t$, flowing for $\tau$ time units.


 Given a base time
$t$ and a flow duration $\tau$, our motivation is to discover \emph{coherent pairs}
of subsets $A_t, A_{t+\tau}\subset M$ such that
$\Phi(A_t,t,\tau)\approx A_{t+\tau}$.
More precisely, we will call $A_t, A_{t+\tau}$ a
\emph{$(\rho_0,t,\tau)-$coherent pair} if
\begin{equation}\label{eq:invariant}
\rho_\mu(A_t,A_{t+\tau}):={\mu(A_t\cap \Phi(A_{t+\tau},t+\tau;
-\tau))}/{\mu(A_t)}\ge \rho_0,
\end{equation}
and $\mu(A_t)=\mu(A_{t+\tau})$, where $\mu$ is a ``reference'' probability measure at time $t$.
The measure $\mu$ describes the mass distribution of the quantity we wish to study the transport of over the interval $[t,t+\tau]$;  $\mu$ need not be invariant under the flow $\Phi$.

We are only interested in coherent pairs that remain coherent under small diffusive perturbations of the flow;  robust coherent pairs.
Clearly, a $(1,t,\tau)$-coherent pair can be produced by choosing an arbitrary $A_t$ and setting $A_{t+\tau}=\Phi(A_t,t;\tau)$.
However, such a pair may not be stable if some diffusion is added to the system. In a chaotic system, the set $A_{t+\tau}=\Phi(A_t,t;\tau)$ defined as above will experience stretching and folding, and for moderate to large $\tau$ will become very thin and geometrically irregular.
A small amount of diffusion will then easily eject many particles from $A_{t+\tau}$, reducing the coherence ratio $\rho_\mu(A_t,A_{t+\tau})$.
The requirement that coherent pairs be robust under diffusive perturbations favors coherent sets that are geometrically regular;  these robust, regular sets are more likely to be more dynamically meaningful than non-robust, irregular sets.

Our basic tool for identifying sets satisfying~\eqref{eq:invariant} is the transfer (or Perron-Frobenius) operator
$\mathcal{P}_{t,\tau}:L^1(M,\ell)\circlearrowleft$ defined by
\begin{equation}\label{eq:pfeqn}
\mathcal{P}_{t,\tau}f(z):=f(\Phi(z,t+\tau;-\tau))\cdot|\det
D\Phi(z,t+\tau; -\tau)|
\end{equation}
where $\ell$ is normalized Lebesgue measure on $M$.
If $f(z)$ is a density of passive tracers at time $t$,
$\mathcal{P}_{t,\tau}f(z)$ is the tracer density at time $t+\tau$
induced by the flow $\Phi$. In the autonomous setting, almost-invariant sets were determined \cite{DJ99,FD03,froyland05} by thresholding
eigenfunctions $f$ of $\mathcal{P}_{t,\tau}$ ($=\mathcal{P}_\tau$
for all $t$) corresponding to positive eigenvalues $\lambda\approx
1$:  $A=\{f<c\}$ or $\{f>c\}$.

The above calculations involved constructing a Perron-Frobenius operator for the action of $\Phi$ on the \emph{entire domain} $M$.
In the time-dependent setting, we wish to study transport from $X\subset M$ to a small neighborhood $Y$ of $\Phi(X,t;\tau)\subset M$.  A global analysis would mean that $X=Y=M$ and a transfer operator would be constructed for all of $M$.  However, often one is interested in the situation where the domain of interest $X$ is ``open'' and trajectories may leave $X$ in a finite time (our numerical examples in Sections \ref{sect:eg1} and \ref{sect:eg2} illustrate this). Moreover, the subset $X$ may be very small in comparison to $M$.  In such instances, there are great computational savings if the analysis can be carried out using a non-global Perron-Frobenius operator defined on $X$ rather than $M$.  \emph{Our new methodology allows precisely this and is a significant theoretical and numerical advance over existing transfer operator numerics.}


We now describe a numerical approximation of the action of $\mathcal{P}_{t,\tau}$ from a space of functions supported on $X$ to a space of functions supported on $Y$.
We subdivide the subsets $X$ and $Y$ into collections of sets $\{B_1,\ldots,B_m\}$ and $\{C_1,\ldots,C_n\}$ respectively.
We construct a
finite-dimensional numerical approximation of the transfer operator $\mathcal{P}_{t,\tau}$, using a modification of Ulam's method~\cite{ulam}:
\begin{equation}\label{eq:measureP}
\mathbf{P}^{(\tau)}(t)_{i,j}=\frac{\ell(B_i\cap\Phi(C_j,t+\tau;-\tau))}{\ell(B_i)},
\end{equation}
where $\ell$ is a normalized volume measure. Clearly, the the matrix $\mathbf{P}^{(\tau)}(t)$ is row-stochastic by its
construction. The value $\mathbf{P}^{(\tau)}(t)_{i,j}$ may be
interpreted as the probability that a randomly chosen point in $B_i$ has its image in $C_j$.
We numerically estimate $\mathbf{P}^{(\tau)}(t)_{i,j}$ by
\begin{equation}\label{eq:numericalulameqn}
\mathbf{P}^{(\tau)}(t)_{i,j}\approx{\#\{r:z_{i,r}\in B_i,
\Phi(z_{i,r},t;\tau)\in C_j\}}/{Q},
\end{equation}
where $z_{i,r}$, $r=1,\ldots,Q$ are uniformly distributed test
points in $B_i(t)$ and $\Phi(z_{i,r},t;\tau)$ is obtained via a
numerical integration.

The numerical discretization has the useful side-benefit of producing
a discretization-induced diffusion with magnitude the order of the image of box diameters (see Lemma 2.2 \cite{froyland95sbr}).
Ultimately, in Section \ref{sect:comp} we will construct coherent sets by thresholding vectors in $\sp\{\chi_{B_1},\ldots,\chi_{B_m}\}$ and $\sp\{\chi_{C_1},\ldots,\chi_{C_n}\}$.
This discretization limits the irregularity of possible coherent sets, and in practice, high regularity is observed.

%
%
%

\section{Coherent partitions}
\label{sect:tech}
For the remainder of the paper we set $P_{ij}=\mathbf{P}^{(\tau)}(t)_{i,j}$, fixing $t$ and $\tau$.
We set $p_i=\mu(B_i), i=1,\ldots,m$ and assume that $p_i>0$ for all $i=1,\ldots,n$ (if some sets $B_i$ have zero reference measure, we remove them from our collection as there is no mass to be transported).
%
Define $q=pP$ to be the image probability vector on $Y$;  we assume $q>0$ (if not, we remove sets $C_j$ with $q_j=0$).  The probability vector $q$ defines a probability measure $\nu$ on $Y$ via $\nu(Y')=\sum_{j=1}^n q_j\ell(Y'\cap C_j)$ for measurable $Y'\subset Y$. We may think of the probability measure $\nu$ as the discretized image of $\mu$.

\subsection{Problem setup}
To find the most coherent pair, we first try to partition $X$ and $Y$ as $X_1\cup X_2$ and $Y_1\cup Y_2$ in a particular way, where $X_1, X_2, Y_1, Y_2$ all have measure approximately 1/2.  This restriction will be relaxed later.
Let $I_1,I_2$ partition $\{1,\ldots,m\}$ and $J_1,J_2$ partition $\{1,\ldots,n\}$ and set $X_k=\cup_{i\in I_k} B_i$ and $Y_k=\cup_{j\in J_k} C_j$, $k=1,2$.
We desire:
\begin{enumerate}
\item $\mu(X_k)=\sum_{i\in I_k}p_i\approx 1/2, \nu(Y_k)=\sum_{j\in J_k}q_i\approx 1/2, k=1,2$

    (the sets $X_1, X_2$ and $Y_1, Y_2$ partition $X$ and $Y$ into two sets of roughly equal $\mu$-mass and $\nu$-mass respectively).
\item $\rho_\mu(X_k,Y_k)\approx 1, k=1,2$

(this is a measure-theoretic way of saying that $\Phi(X_k,t;\tau)\approx Y_k$, $k=1,2$).
\end{enumerate}

\subsection{Solution approach}
Introduce the inner products $\langle x_1,x_2\rangle_p=\sum_i x_{1,i}x_{2,i}p_i$ and $\langle y_1,y_2\rangle_q=\sum_j y_{1,j}y_{2,j}q_j$.
We form a normalized matrix $L_{ij}=p_iP_{ij}/q_j$. The $L$ resulting from this normalization of $P$ ensures that $\mathbf{1}L=\mathbf{1}$.  We think of $L$ as a transition matrix from the inner product space $(\mathbb{R}^m,\langle\cdot,\cdot\rangle_p)$ to the inner product space $(\mathbb{R}^n,\langle\cdot,\cdot\rangle_q)$, which takes a uniform density on $(\mathbb{R}^m,\langle\cdot,\cdot\rangle_p)$ (representing the measure $\mu$) to a uniform density on $(\mathbb{R}^n,\langle\cdot,\cdot\rangle_q)$ (representing the measure $\nu$).

To describe 2-partitions of $X$ and $Y$ we consider vectors $x\in\{\pm 1\}^m, y\in \{\pm 1\}^n$ and define $X_1=\bigcup_{i:x_i=1}B_i, X_2=\bigcup_{i:x_i=-1}B_i, Y_1=\bigcup_{i:y_i=1}C_i, Y_2=\bigcup_{i:y_i=-1}C_i$.
Thus, the partitions $I_1,I_2$ and $J_1,J_2$ are described by the parity of $x$ and $y$ respectively.
We can write the condition $\mu(X_k)=\sum_{i\in I_k}p_i\approx 1/2, \nu(Y_k)=\sum_{j\in J_k}q_i\approx 1/2, k=1,2$ as $|\langle x,\mathbf{1}\rangle_p|, |\langle y,\mathbf{1}\rangle_q|<\epsilon$ for small $\epsilon$ (the $\epsilon$ is needed as it may be impossible to form finite collections of sets $B_i$ and $C_j$ with measure exactly 1/2).

Consider the problem:
\begin{equation}
\label{combinatorial}
\max\{\langle xL,y\rangle_q: x\in\{\pm 1\}^m, y\in \{\pm 1\}^n, |\langle x,\mathbf{1}\rangle_p|, |\langle y,\mathbf{1}\rangle_q|<\epsilon\}
\end{equation}
for some small $\epsilon$.

The objective
\begin{eqnarray*}\label{eq:objective}
\langle xL,y\rangle_q&=&\left(\sum_{i\in I_1,j\in J_1} L_{ij}q_j+\sum_{i\in I_2,j\in J_2} L_{ij}q_j\right)-\left(\sum_{i\in I_1,j\in J_2} L_{ij}q_j+\sum_{i\in I_2,j\in J_1} L_{ij}q_j\right)\\
&=&\left(\sum_{i\in I_1,j\in J_1} p_iP_{ij}+\sum_{i\in I_2,j\in J_2} p_iP_{ij}\right)-\left(\sum_{i\in I_1,j\in J_2} p_iP_{ij}+\sum_{i\in I_2,j\in J_1} p_iP_{ij}\right)\\
&\approx&\left(\mu(X_1\cap \Phi(Y_1,t+\tau; -\tau))+\mu(X_2\cap \Phi(Y_2,t+\tau; -\tau))\right)\\
&\quad&-\left(\mu(X_1\cap \Phi(Y_2,t+\tau; -\tau))+\mu(X_2\cap \Phi(Y_1,t+\tau; -\tau))\right)\\
&=&\mu(X_1)\rho_\mu(X_1,Y_1)+\mu(X_2)\rho_\mu(X_2,Y_2)-\mu(X_1)\rho_\mu(X_1,Y_2)-\mu(X_2)\rho_\mu(X_2,Y_1).
\end{eqnarray*}
Thus, maximizing $\langle xL,y\rangle_q$ is a very natural way to achieve our aim of finding partitions so that $\rho_\mu(X_k,Y_k)\approx 1, k=1,2$.
The approximation in the above reasoning occurs because
$P_{ij}\approx \mu(B_i\cap\Phi(C_j,t+\tau;-\tau))/\mu(B_i)$\footnote{If $\mu$ is absolutely continuous with a positive density that is Lipschitz on the interior of each $B_i$, then this error goes to zero with decreasing diameter of $B_i$ and $C_j$;  see Lemma 3.6 \cite{froyland98}.}.

The problem (\ref{combinatorial}) is a difficult combinatorial problem;  as a heuristic means of finding a good solution we relax the binary restriction on $x$ and $y$ and allow them to take on continuous values.
We will interpret the values of $x$ and $y$ as ``fuzzy inclusions'';  if $x_i$ is very positive, then $B_i$ is very likely to belong to $X_1$, and if $x_i$ is very negative, then $B_i$ is very likely to belong to $X_2$.
Similarly for strong positivity or negativity of $y_i$ and inclusion of $B_i$ in $Y_1$ or $Y_2$ respectively.
If the value of $x_i$ or $y_i$ is near to zero, the fuzzy inclusion is less certain and we use an optimization in Algorithm \ref{alg1} (Section 3.3) to determine where $B_i$ belongs.

As $x$ and $y$ can now float freely, we can set $\epsilon=0$, and thus may insist that $\langle x,\mathbf{1}\rangle_p=\langle y,\mathbf{1}\rangle_q=0$.
When restricting $x$ and $y$ to be elements of $\{\pm 1\}^m$ and $\{\pm 1\}^n$, we implicitly set the norms $\|x\|_p=\langle x,x\rangle_p^{1/2}$ and $\|y\|_q=\langle y,y\rangle_q^{1/2}$ to both be 1.
Now that we let $x$ and $y$ freely float, we must include normalization terms in our objective.
Thus, the relaxed problem is
\begin{equation}
\label{relaxed}
\max_{x\in\mathbb{R}^m, y\in\mathbb{R}^n}\left\{\frac{\langle xL,y\rangle_q}{\|x\|_p\|y\|_q}: \langle x,\mathbf{1}\rangle_p=\langle y,\mathbf{1}\rangle_q=0\right\}
\end{equation}
We will use the optimal $x$ and $y$ to create our partition $X_1, X_2$ and $Y_1, Y_2$ via $X_1=\bigcup_{i:x_i>b}B_i, X_2=\bigcup_{i:x_i<b}B_i, Y_1=\bigcup_{i:y_i>c}C_i, Y_2=\bigcup_{i:y_i<c}C_i$, where $b$ and $c$ are chosen so that $\sum_{i\in I_k}p_i\approx 1/2 \sum_{j\in J_k}q_i\approx 1/2, k=1,2$.
As an extension to our heuristic, we may also relax the condition that the measures of $X_1, X_2$ and $Y_1, Y_2$ are all approximately 1/2, and only enforce $\mu(X_k)=\sum_{i\in I_k}p_i\approx \sum_{j\in J_k}q_i=\nu(Y_k), k=1,2$. This would mean that while there is some flexibility in the choice of $b$, the value $c$ is a function of $b$;  see Algorithm \ref{alg1}.

We close this section with a lemma stating the solution to (\ref{relaxed}).
\label{sect:comp}
\begin{lemma}
\label{complemma}
Let $\Pi_p$ be an $m\times m$ diagonal matrix with $p$ on the diagonal and $\Pi_q$ be an $n\times n$ diagonal matrix with $q$ on the diagonal.  Suppose that $PP^\top$ is an irreducible matrix\footnote{there exists a $k$ such that $(PP^\top)^k>0$.}.
The value of (\ref{relaxed}) is $\sigma_2$, the second largest singular value of $\Pi_p^{1/2}P\Pi_q^{-1/2}$, and the maximizing $x$ and $y$ in (\ref{relaxed}) are given by $x=\hat{x}\Pi_p^{-1/2}$ and $y=\hat{y}\Pi_q^{-1/2}$, where $\hat{x}$ and $\hat{y}$ are the corresponding left and right singular vectors.
\end{lemma}
\begin{proof}
See appendix.
\end{proof}

\subsection{Extraction of coherent pairs}
We now detail the procedure that extracts the coherent pairs $X_k, Y_k$ from the vectors $x$ and $y$ identified in Lemma \ref{complemma}.
We create sets that are unions of boxes with $x$ and $y$ values above certain thresholds. Define
${X}_1(b):=\bigcup_{i:x_i>b}B_i$ and
${Y}_{1}(c):=\bigcup_{j:y_j>c}C_j$,
$b,c\in\mathbb{R}$.
Define
\begin{equation}\label{eq:coherence}
\tilde{\rho}(\tilde{X}_1(b),\tilde{Y}_1(c))=\frac{\sum_{i:B_i\subset \tilde{X}_1(b),j:C_j\subset \tilde{Y}_1(c)}p_iP_{ij}}{\sum_{i:B_i\subset \tilde{X}_1(b)}p_i}.
\end{equation}

The quantity $\tilde{\rho}$ measures the discretized coherence for the pair
$\tilde{X}_1(b), \tilde{Y}_1(c)$. Our procedure to vary the thresholds $b$ and $c$ so as to select $\tilde{X}_1(b)$ and $\tilde{Y}_1(c)$ with largest $\tilde{\rho}$ value is summarized below:
\begin{algorithm}
\label{alg1}
\begin{enumerate}\quad \label{enum:algorithm}
\item Let $\eta(b)=\arg\min_{c'\in\mathbb{R}}\bigl|
\mu(\tilde{X}_1(b))-\nu(\tilde{Y}_1(c'))\bigr|$.
This is to make $\nu(\tilde{Y}_1(c'))$ as close as possible to $\mu(\tilde{X}_1(b))$.
\item Set $b^*=\arg\max\tilde\rho(\tilde{X}_1(b),\tilde{Y}_1(\eta(b)))$. The value of $b^*$
is selected to maximize the coherence.
\item Define
$A_t=X_1:=\tilde{X}_1(b^*)$ and
$A_{t+\tau}=Y_1:=\tilde{Y}_1(\eta(b^*))$.
\end{enumerate}
\end{algorithm}
To obtain $X_2$ and $Y_2$, we define $X_2=\tilde{X}_2(b^*):=\bigcup_{i:x_i\le b^*}B_i$ and
$Y_2=\tilde{Y}_{2}(\eta(b^*)):=\bigcup_{j:y_j\le \eta(b^*)}C_i$, the complements of $X_1$ and $Y_1$ in $X$ and $Y$, respectively.
Thus, we select $X_1$ and $Y_1$ to be the most coherent pair and define $X_2$ and $Y_2$ as their respective complements.
One now should repeat Algorithm \ref{alg1} with $\tilde{X}_2(b):=\bigcup_{i:x_i\le b}B_i$ and
$\tilde{Y}_{2}(c):=\bigcup_{j:y_j\le c}C_j$,
$b,c\in\mathbb{R}$ in place of $\tilde{X}_1(b)$ and $\tilde{Y}_1(c)$ to search from ``the negative end'' of the vectors $x$ and $y$, possibly picking up a pair with higher coherence, and defining $X_1, Y_1$ as the complements of $X_2, Y_2$.

%
%
%
%
%
%
\section{Example 1: Idealized Stratospheric flow}
\label{sect:eg1}
We consider the Hamiltonian system $\frac{dx}{dt}=-\frac{\partial\Phi}{\partial y},\frac{dy}{dt}=\frac{\partial\Phi}{\partial x}$ where
\begin{equation}\label{eq:idealizedflow}
\begin{split}
&\Phi(x,y,t)=c_3y-U_0L\tanh(y/L)+A_3U_0L{\mathrm{sech}}^2(y/L)\cos(k_1x)\\
&+A_2U_0L{\mathrm{sech}}^2(y/L)\cos(k_2x-\sigma_2t)+A_1U_0L{
\mathrm{sech}}^2(y/L)\cos(k_1x-\sigma_1t).\\
\end{split}
\end{equation}
This quasi-periodic system represents an idealized stratospheric flow in the northern or southern hemisphere. Rypina \emph{et al.}~\cite{Rypina_etal07} show that there is a time-varying jet core oscillating in a band around $y=0$ and three Rossby waves in each of the regions above and below the jet core. The parameters studied in \cite{Rypina_etal07} are chosen so that the jet core forms a complete transport barrier between the two Rossby wave regimes above and below it. We modify some of the parameters to remove the jet core band and allow transport between the two Rossby wave regimes. We expect that the two Rossby wave regimes will form time-dependent coherent sets because transport between the two regimes is considerably less than the transport within regimes.We set the parameters as follows: $c_2/U_0=0.205$, $c_3/U_0=0.700$, $A_3=0.2, A_2=0.4$ and $A_1=0.075$, with the remaining parameters as stated in Rypina \emph{et al.}~\cite{Rypina_etal07}.


Our initial time is $t=20$ days and our final time is $t+\tau=30$ days. At our initial time we set $X=S^1\times[-2.5, 2.5]$ Mm, where $S^1$ is a circle parameterised from 0 to $6.371\pi$ Mm, and subdivide $X$ into a grid of $m=28200$ identical boxes $X=\{B_1,\ldots,B_m\}$.
This choice of $m$ is sufficiently large to represent the dynamics to a good resolution.
We compute an approximation of $\Phi(X,20;30)$ by uniformly distributing $Q=400$ sample points in each grid box and numerically calculating $\Phi(z_{i,r},20;30)$ using the standard Runge-Kutta method.
The choice of $Q$ is made so that over the flow duration, the image of boxes is well represented by the $Q$ sample points per box.
These $Q\times m$ image points are then covered by a grid of $n=34332$ boxes $\{C_1,\ldots,C_n\}$ of the same size as the $B_i$, $i=1,\ldots,m$.

We set $Y=\bigcup_{j=1}^n C_j$, covering the approximate image of $X$.
The transition matrix $P=\mathbf{P}^{(30)}_{20}$ is computed using (\ref{eq:numericalulameqn}).

As the flow is area preserving, a natural reference measure $\mu$ is Lebesgue measure, which we normalize so that $\mu(X)=1$. Thus, $\mu(B_i)=p_i=1/m$, $i=1,\ldots,m$ and so $(\Pi_p)_{ii}=1/m$, $i=1,\ldots,m$. The vector $q$ is constructed as $q=pP$.
We compute the second largest singular value of $\Pi_p^{1/2}P\Pi_q^{-1/2}$ and the corresponding left and right singular vectors and thus determine $x$ and $y$ from Lemma \ref{complemma}. The top two singular values were computed to be $\sigma_1=1.0$ and $\sigma_2\approx0.996$.
We expect $x$ to determine coherent sets at time $t=20$ days and $y$ to determine coherent sets at time $t+\tau=30$ days.
Figure~\ref{fig:subfig1} and~\ref{fig:subfig2} illustrate the vectors $x$ and $y$, which provide clear separations into red (positive) and green (mostly negative) regions.
\begin{figure}[]\label{fig:compareSVSUFTLE}
\centering
\subfigure[]{
\includegraphics[scale=.1975]{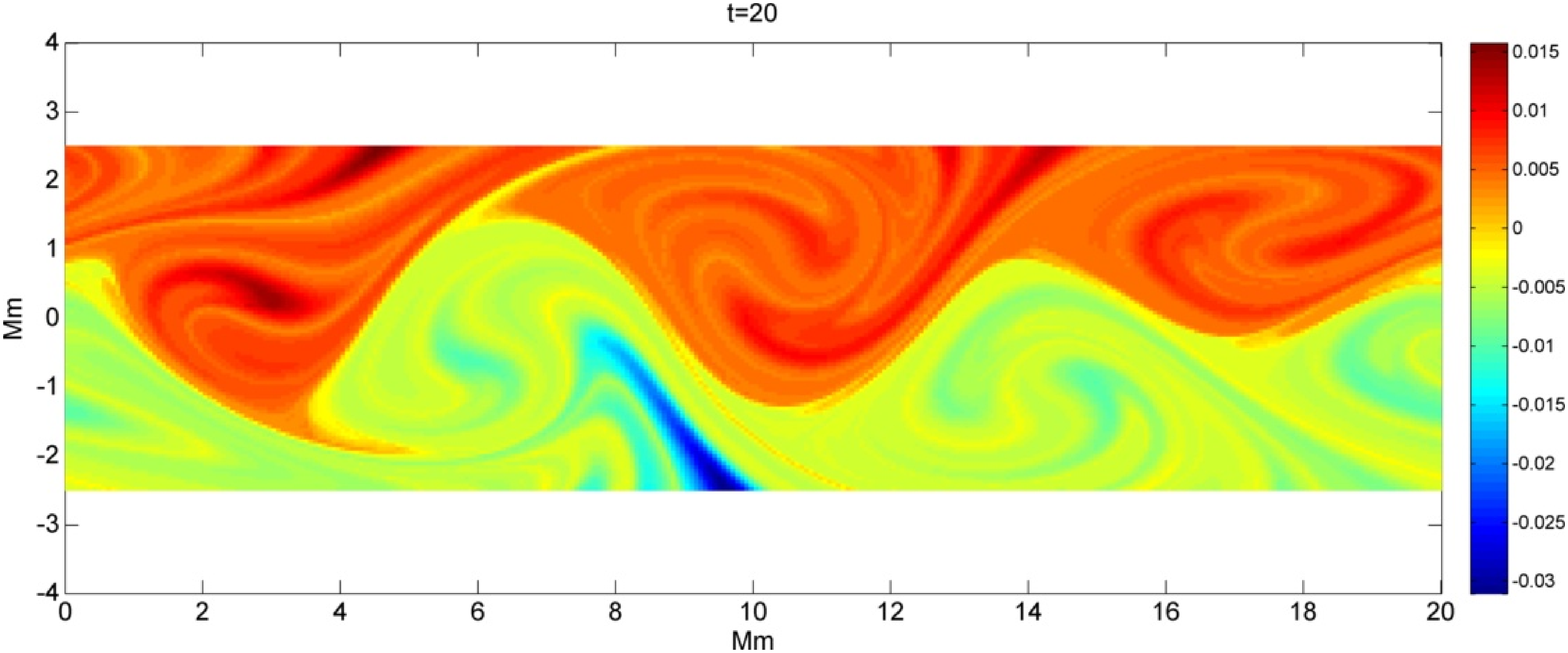}
\label{fig:subfig1}
}
\subfigure[]{
\includegraphics[scale=.1975]{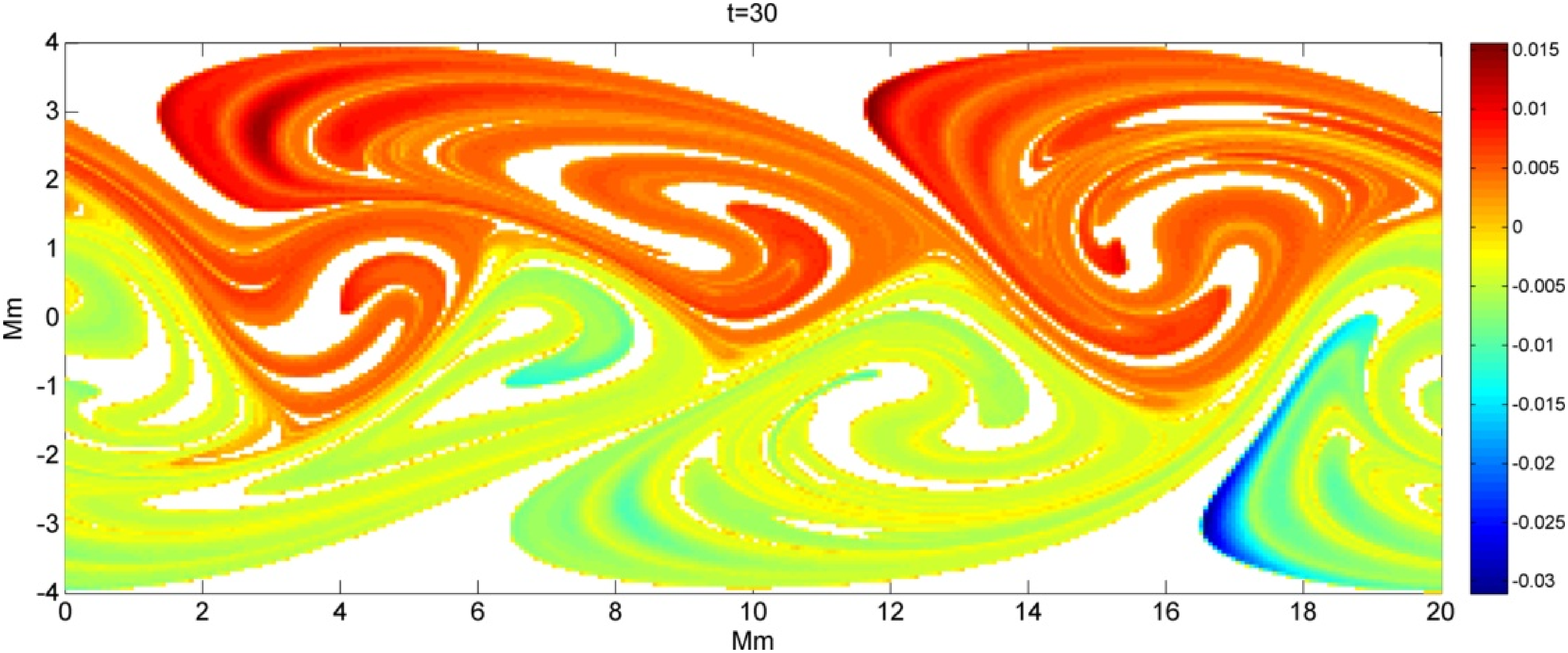}
\label{fig:subfig2}
}
\subfigure[]{
\includegraphics[scale=.2]{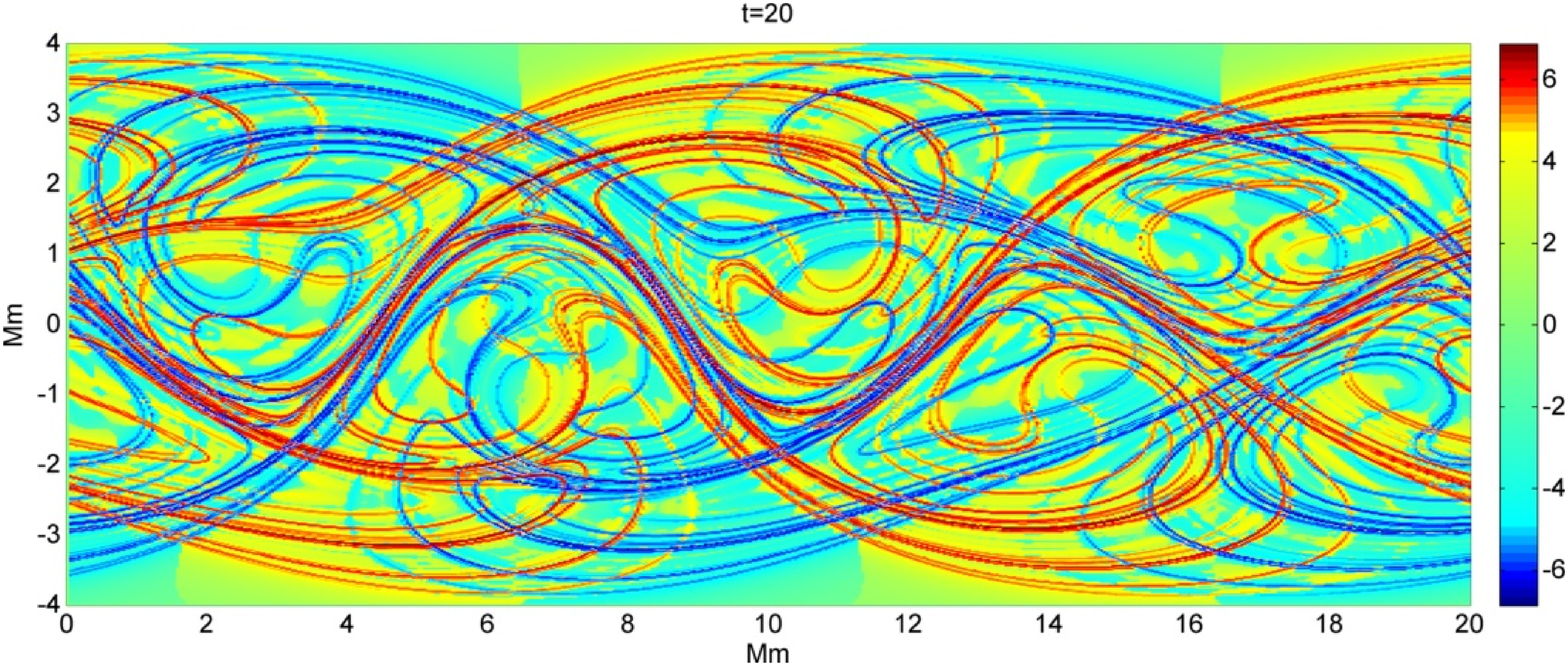}
\label{fig:subfig3}
}
\subfigure[]{
\includegraphics[scale=.2]{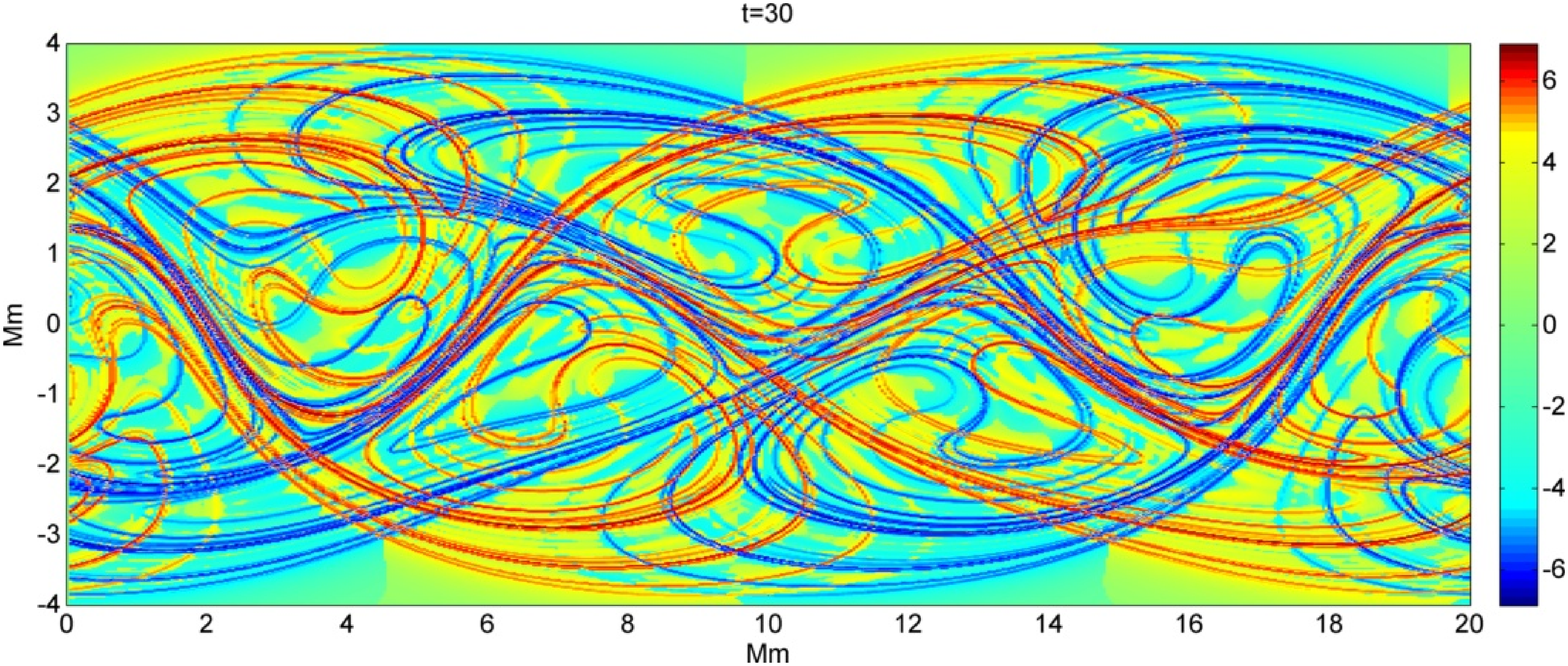}
\label{fig:subfig4}
}
\caption[]{(a) The optimal vector $x$; (b) the optimal vector $y$; (c) Backward-time (blue) and Forward-time (red) FTLEs at $t=20$ days computed with the flow time $\tau=10$ days; (d) Backward-time (blue) and Forward-time (red) FTLEs at $t=30$ days computed with the flow time $\tau=10$ days.}
\end{figure}

We apply the thresholding Algorithm~\ref{alg1} to the vectors $x$ and $y$ to obtain the pairs $(X_1,Y_1)$ and $(X_2,Y_2)$\footnote{When determining $X_1$ and $Y_1$, Algorithm \ref{alg1} produced values $b^*\approx0.0077$ and $\eta(b^*)\approx0.0005$.} shown in Figures \ref{fig:subfig11} and \ref{fig:subfig12}.
To demonstrate that $Y_1\approx \Phi(X_1,20;10)$, we plot the latter set in Figure \ref{fig:subfig13}.
When compared with Figure \ref{fig:subfig12} we see that there is very little leakage from $Y_1$, just a few thin filaments.
Similarly, Figures \ref{fig:subfig14} and \ref{fig:subfig12} compare $Y_2$ and $\Phi(X_2,20;10)$, again showing a small amount of leakage.
This leakage is quantified by computing $\tilde{\rho}(X_1,Y_1)\approx\tilde{\rho}(X_2,Y_2)\approx0.98$.

%

\begin{figure}[]\label{fig:Interface}
\centering
\subfigure[]{
\includegraphics[scale=.175]{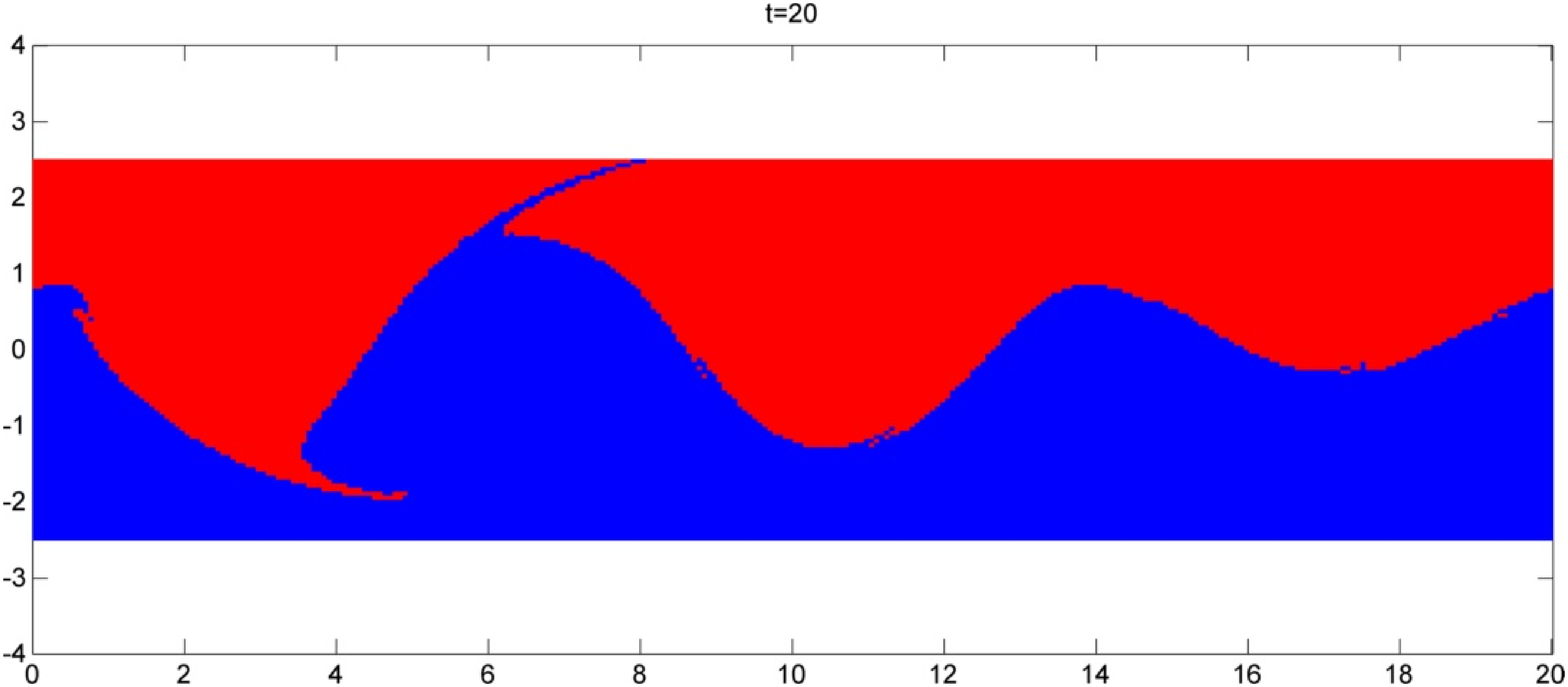}
\label{fig:subfig11}
}
\subfigure[]{
\includegraphics[scale=.175]{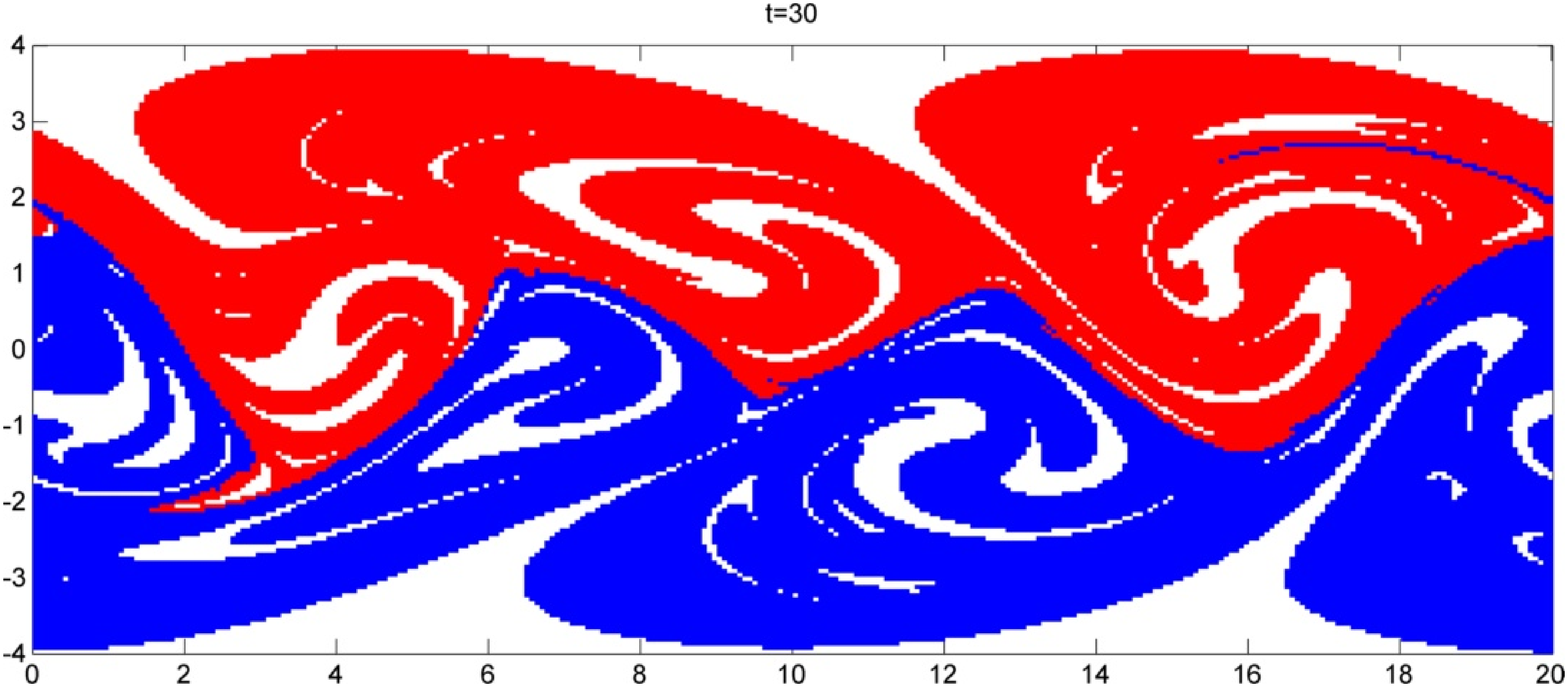}
\label{fig:subfig12}
}
\subfigure[]{
\includegraphics[scale=.175]{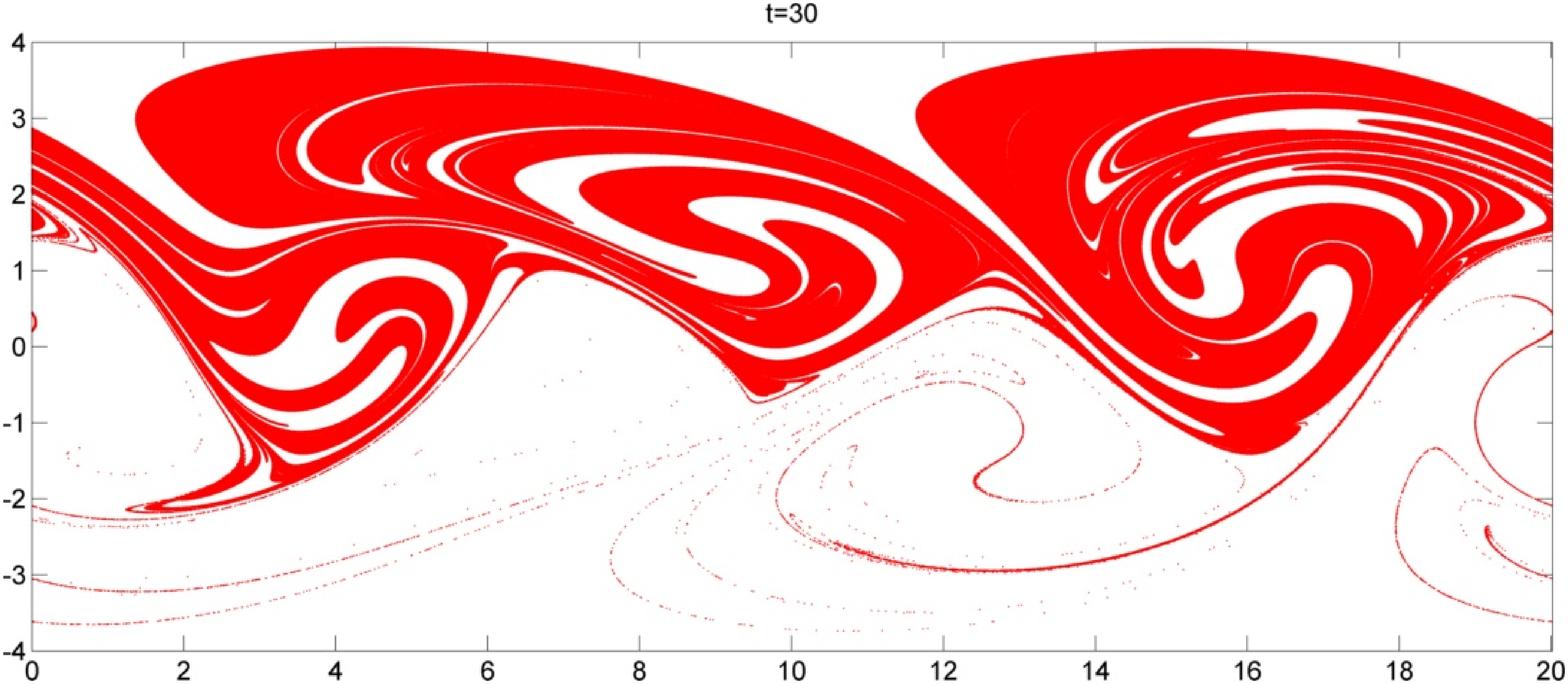}
\label{fig:subfig13}
}
\subfigure[]{
\includegraphics[scale=.175]{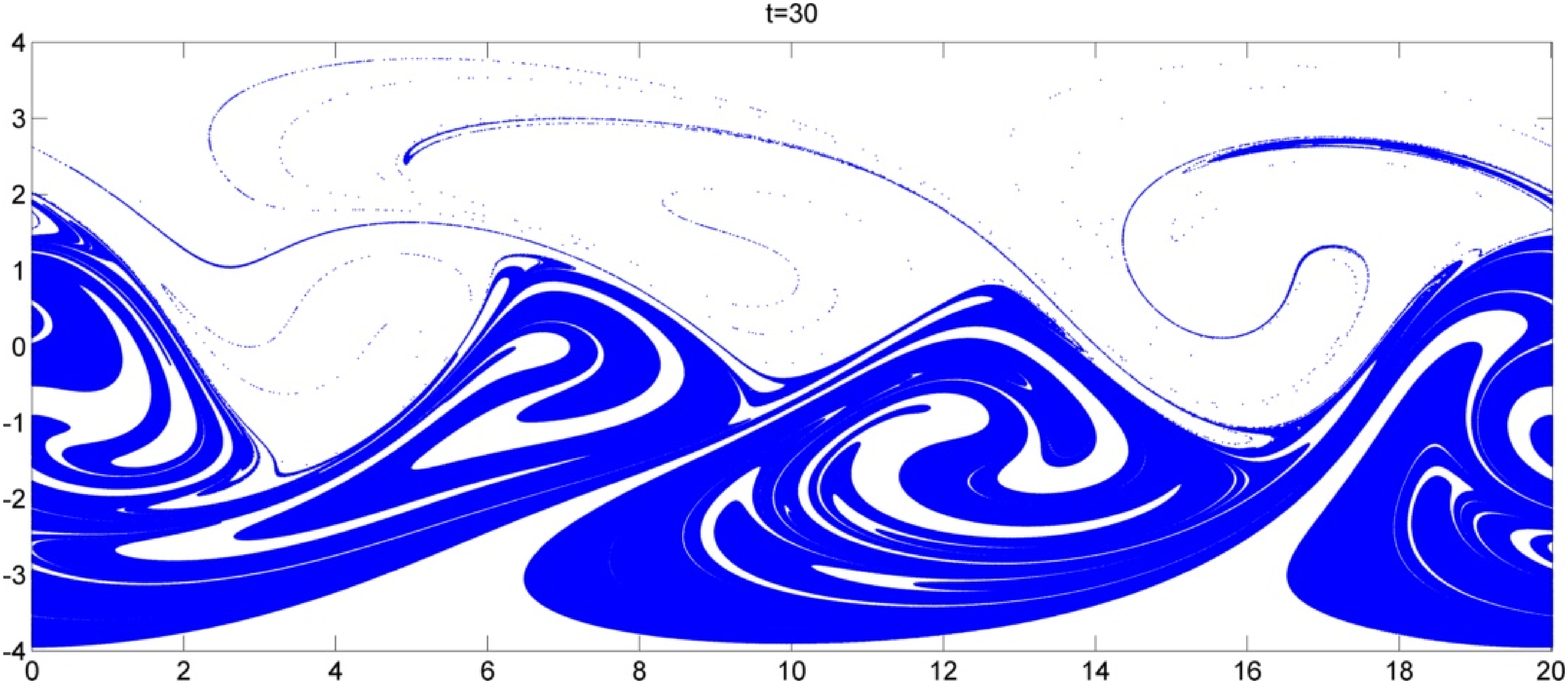}
\label{fig:subfig14}
}
\caption[]{(a) The sets $X_1$ (red) and $X_2$ (blue); (b) the sets $Y_1$ (red) and $Y_2$ (blue); (c) the set $\Phi(X_1,20;10)$; (d) the set $\Phi(X_2,20;10)$.}
\end{figure}


We compare our results with the attracting and repelling material lines
computed via the finite-time Lyapunov exponent (FTLE) field~\cite{haller00} with the flow time $\tau=10$.
The ridges of the FTLE fields are commonly used to identify barriers to transport.
Figures~\ref{fig:subfig3} and~\ref{fig:subfig4} present an overlay of forward- and backward-time FTLEs at $t=20$ and $t=30$, respectively.
In this example, there are several FTLE ridges in the vicinity of the dominant transport barrier across the middle of the domain, and also several ridges far away from this barrier.  The FTLE ridges do not crisply and unambiguously identify the dominant transport barrier shown in Figures \ref{fig:subfig1} and \ref{fig:subfig2}.


\section{Example 2: Stratospheric polar vortex as coherent sets}
\label{sect:eg2}

In our second example, we use velocity fields obtained from the ECMWF Interim data set (http://data.ecmwf.int/data/index.html).
We focus on the stratosphere over the southern hemisphere south of 30 degrees latitude.
In this region, there are strong persistent transport barriers to midlatitude mixing during the austral winter;  these barriers give rise to the Antarctic polar vortex.  We will apply our new methodology to the ECMWF vector fields in two and three dimensions to resolve the polar vortex as a coherent set.
\subsection{Two dimensions}
Our input data consists of two-dimensional velocity fields on a $121
\times 240$ element grid in the longitude and latitude directions, respectively.
The ECMWF data provides updated velocity fields every 6 hours.
The flow is initialised at September 1, 2008 on a 475K isentropic surface and we follow the flow until September 14.
To a good approximation isentropic surfaces are close to invariant over a period about two weeks~\cite{JL01}.

We set $X= S^{1}\times[-90\degree, -30\degree]$, where $S^{1}$ is a
circle parameterized from $0\degree$ to $360\degree$.
The domain $X$ is initially subdivided into the grid boxes $B_i$, $i=1,\ldots,m$, where $m=13471$ in this example.
Based on the hydrostatic balance and the ideal gas law, we set the reference measure $p_i=\mathrm{Pr}_i^{5/7}\mathrm{a}_i$ for all $i=1,\ldots,m$, where $\mathrm{Pr}_i$ is the pressure at the center point of $B_i$ and $\mathrm{a}_i$ is the area of box $B_i$.

Using $Q=100$ sample points $z_{i,r}$, $r=1,\ldots,Q$ uniformly distributed in each grid box $B_i$, $i=1,\ldots,m$ we calculate an approximate image $\Phi(X,t;\tau)$\footnote{We use the standard Runge-Kutta method with
step size of $3/4$ hours. Linear interpolation is used to evaluate
the velocity vector of a tracer lying between the data grid points
in the longitude-latitude coordinates. In the temporal
direction the data is independently affinely interpolated.} and cover this approximate image with $m=14395$ boxes $\{C_1,\ldots,C_n\}$ to produce the image domain $Y$.
We construct $P=\mathbf{P}^{(t)}_{\tau}$ as described earlier using the same $Q\times m$ sample points.

We compute $x$ and $y$ as described in Lemma \ref{complemma};  graphs of these vectors are shown in Figure~\ref{fig:2D} (Upper left and Upper right).
Figure \ref{fig:2D} (Lower left and Lower right) shows the result of Algorithm \ref{alg1}, extracting coherent sets $A_t$ and $A_{t+\tau}$ from the vectors $x$ and $y$.
We calculate the coherent ratio $\rho_\mu(A_t,A_{t+\tau})\approx 0.991$, which means that 99.1\% of the mass in $A_t$ (September 1, 2008) flows into $A_{t+\tau}$ (September 14, 2008), demonstrating a very high level of coherence.
\begin{figure}[hbt]
\centering
\subfigure{\includegraphics[width=0.325\textwidth]{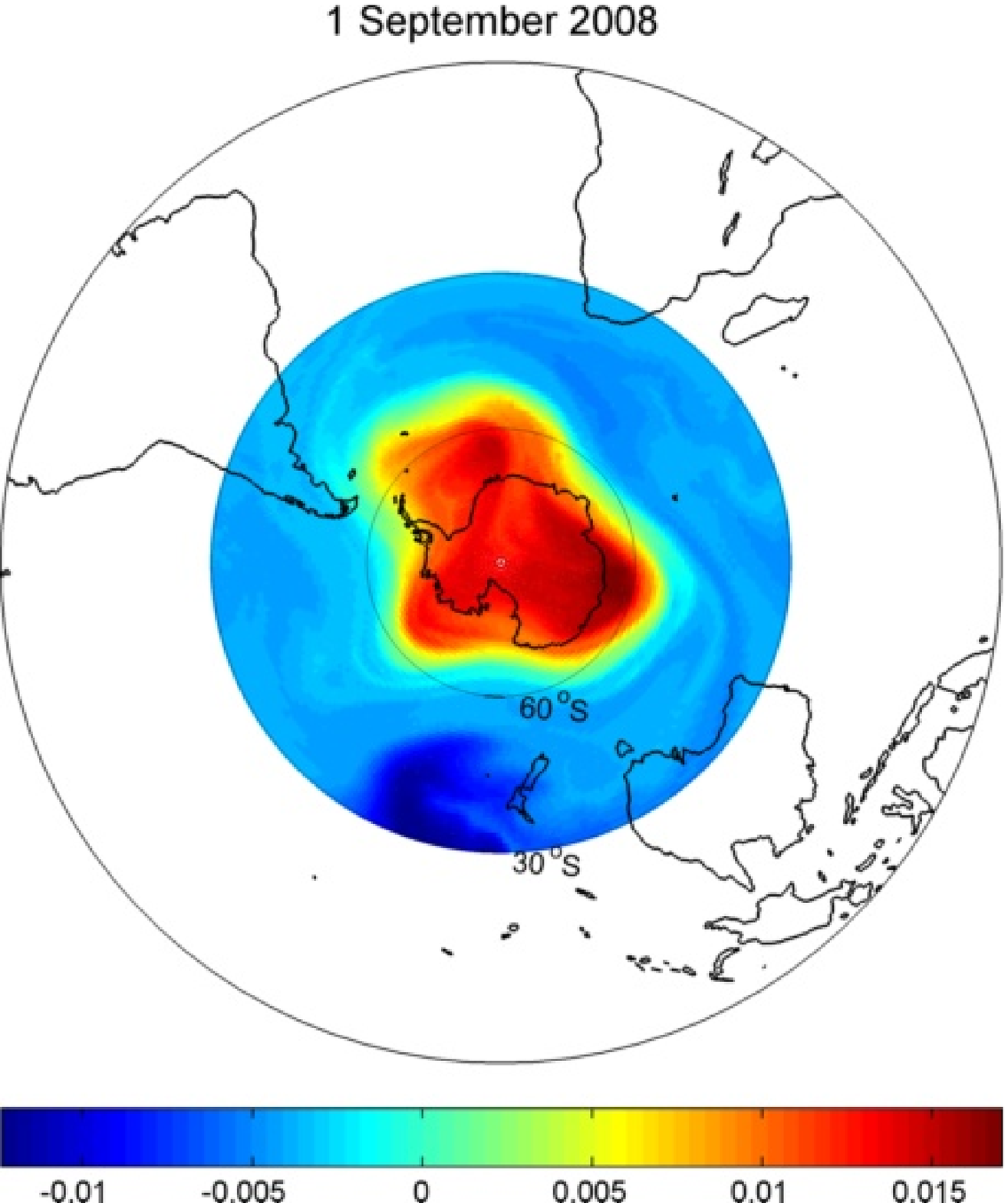}}\qquad\qquad
\subfigure{\includegraphics[width=0.325\textwidth]{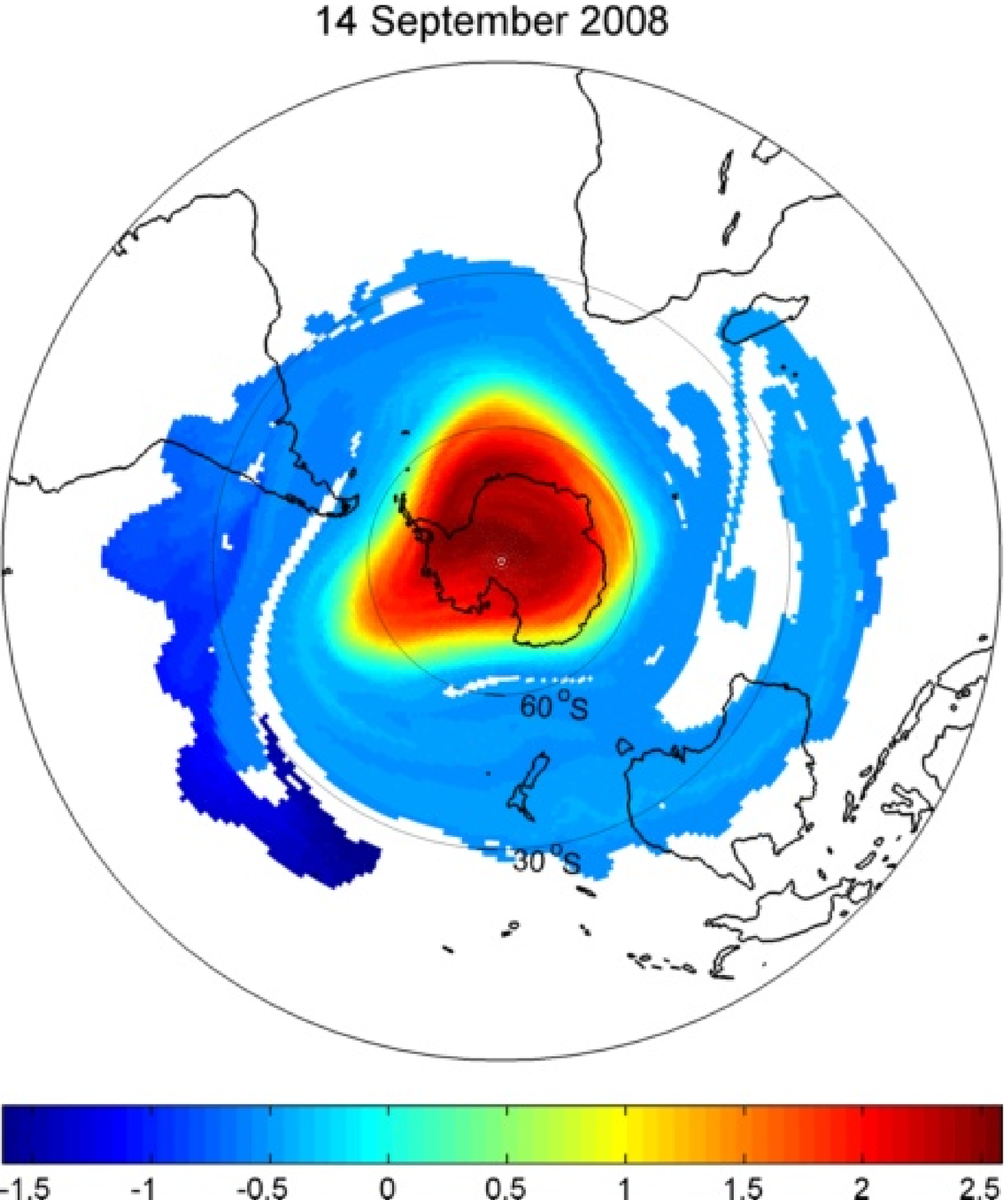}}\qquad\qquad
\subfigure{\includegraphics[width=0.325\textwidth]{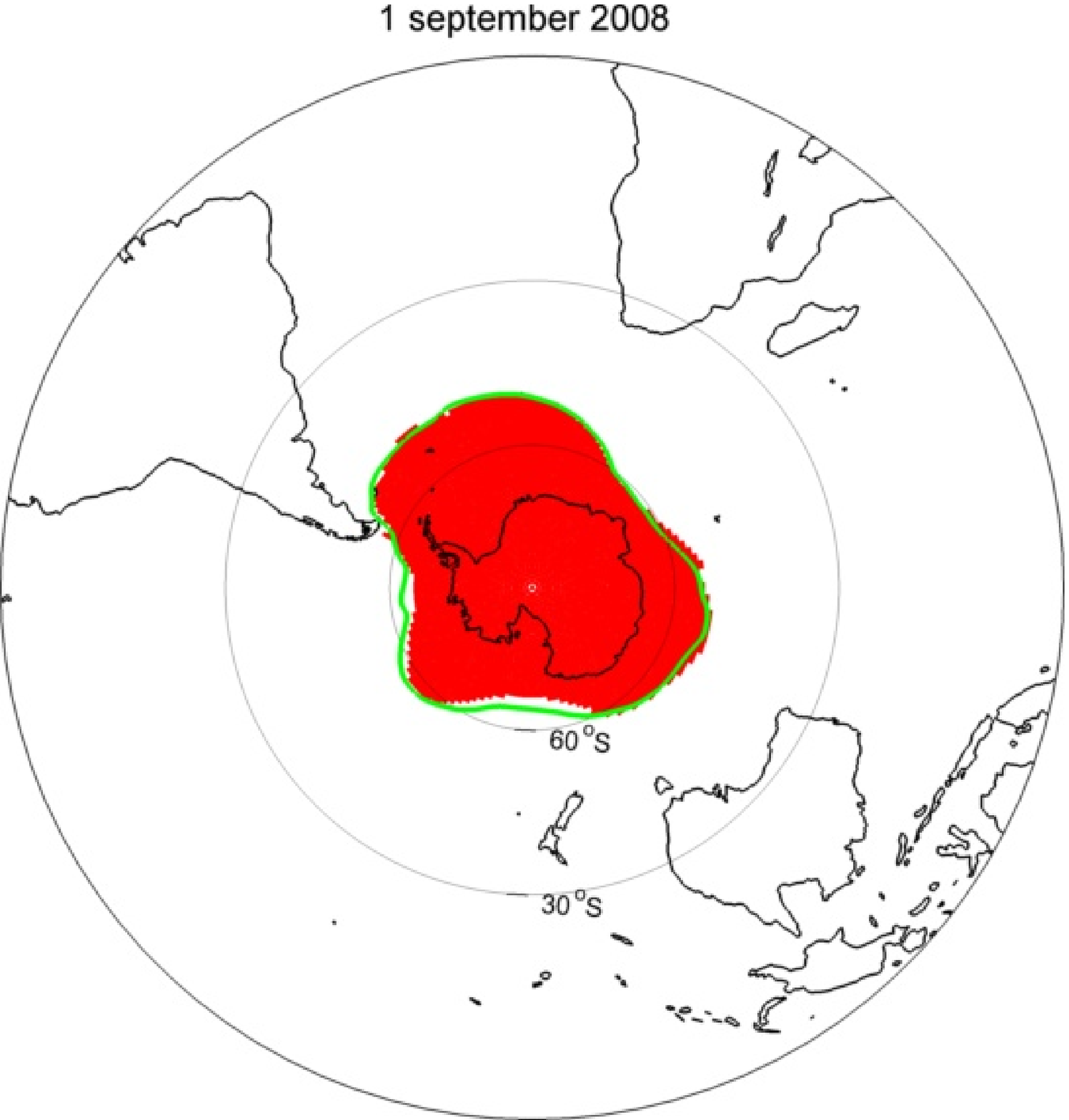}}\qquad\qquad
\subfigure{\includegraphics[width=0.325\textwidth]{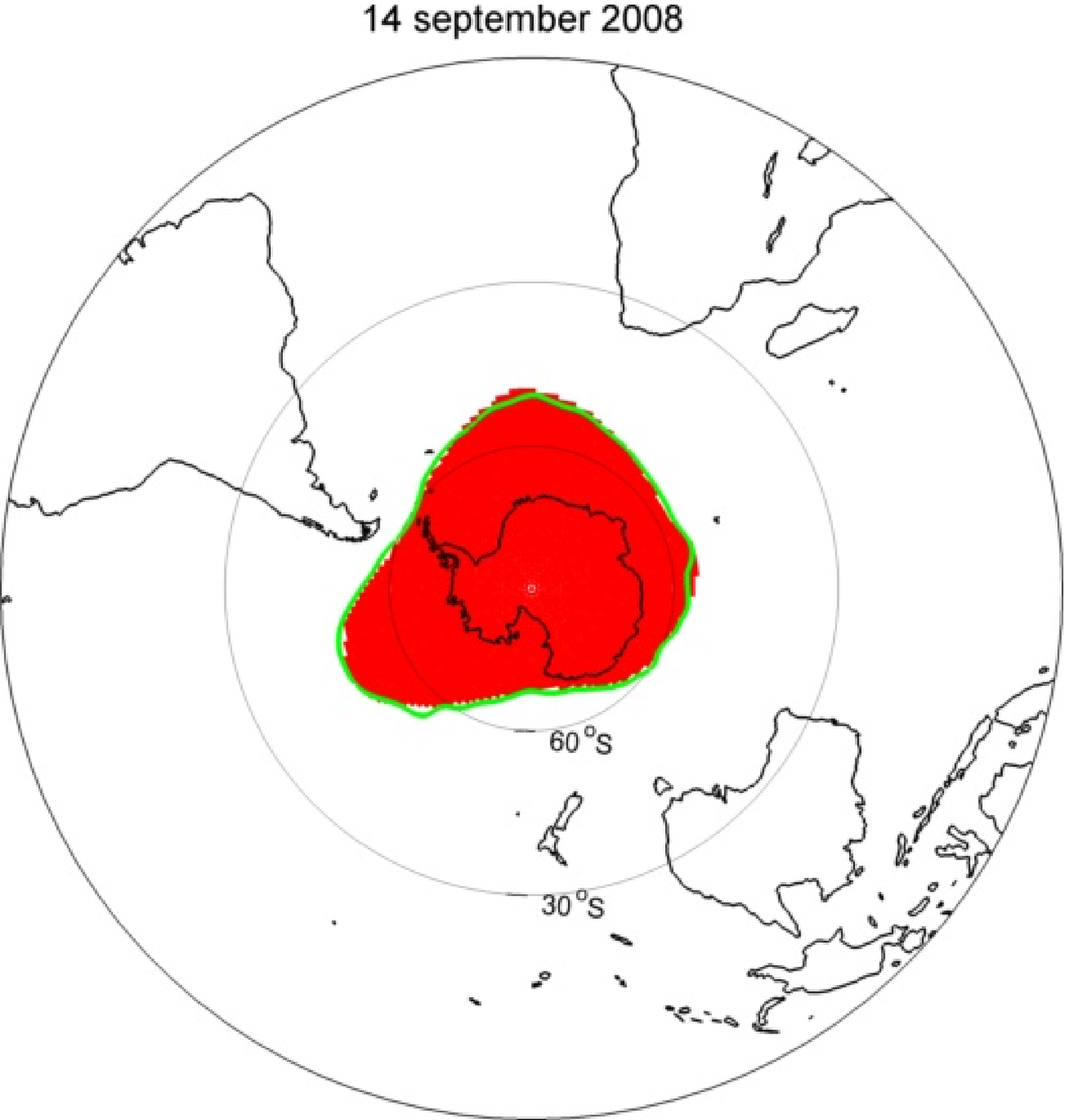}}\qquad\qquad
     \caption{\label{fig:2D}\footnotesize{[Upper left]: Graph of $x$ (September 1, 2008). [Upper right]: Graph of $y$ (September 14, 2008). [Lower left]: The red set represents the coherent set $A_{t}$ (September 1, 2008) obtained from Algorithm \ref{alg1}. The green curve illustrates the vortex edge as estimated using PV. [Lower right]: As for [Lower left] at September 14, 2008.}}
\end{figure}

To benchmark our new methodology, we will compare our result with a method commonly used in the atmospheric sciences to delimit the ``edge'' of the vortex.
It has been recognized that during the winter a strong gradient of
potential vorticity (PV) in the polar stratosphere is developed due to (1)
strong mixing in the mid-latitudes (resulting from the breaking of Rossby
waves emerging from the troposphere and breaking in the stratospheric
"surf zone" ~\cite{MP83}) and (2) weak mixing in the vortex region.
While potential vorticity depends only on the instantaneous vector field, potential vorticity is materially conserved for adiabatic, inviscid flow (both of which are good approximations in stratospheric flow over timescales of a week or two).
Thus, PV may be viewed as a quantity derived from the Lagrangian specification of the flow and  is therefore a meaningful comparator for these nonautonomous experiments who also use Lagrangian information.
We used the method of Sobel et al.~\cite{SPW97} to calculate a PV-based estimate of the vortex edge.
The result is shown by the green curve in Figure \ref{fig:2D} (Lower left and Lower right).
Notating the area enclosed by the green curve at September 1, 2008 by $A^{PV}_t$ and at September 14, 2008 by $A^{PV}_{t+\tau}$, we compute $\rho_\mu(A^{PV}_t,A^{PV}_{t+\tau})\approx 0.984$;  98.4\% of the mass in $A^{PV}_{t}$ flows into $A^{PV}_{t+\tau}$ over the 13 day period.

Our transfer operator methodology is clearly consistent with the accepted potential vorticity approach and in fact identifies a region that experiences slightly \emph{greater} transport barriers across its boundary, indicated by the slightly larger coherence ratio:  99.1\% versus 98.4\%.
In the next section we apply our methodology in three dimensions to estimate the three-dimensional structure of the vortex.

\subsection{Three dimensions}
Strong transport barriers to midlatitude mixing in the southern hemisphere are also known to exist even in the full 3D case, where strong descent occurs near the edges of polar vortex at each pressure
altitude~\cite{Plumb02,Schoeberl_etal92}.
In principle, PV-based methods could be extended to three-dimensions by (i) slicing the three-dimensional region of interest into several nearby isentropic surfaces, (ii) applying the PV methodology on each individual isentropic surface to obtain an estimate of the vortex boundary on that surface, and (iii) stitching together these curves to form a reasonable two-dimensional surface, with the hope that the surface represents an estimate of the boundary of the three-dimensional vortex.
This stitching together of several curves is a nontrivial computational task and complicated geometries may be missed by this relatively simple construction.
The PV approach is likely to be more susceptible to noise than our direct approach because the computation of PV relies on estimates of derivatives of the velocity field (vorticity is the curl of the velocity field).
Finally, such an approach would not utilise the full three-dimensional vector field, but rather a series of vector fields on isentropic surfaces.

A key point of our new methodology is that it can easily applied in either two or three dimensions and works directly with the velocity fields to compute coherent regions with minimal external flux.

We set $X=S^{1}\times[-90\degree, -30\degree]\times[50,70]$, where the third (vertical) component of this direct product is in units of hPa. The ECMWF data is again provided on a $240\times 121$ grid in the longitude/latitude directions, and additionally at 7 pressure levels between 20 and 150 hPa.  We use the full 3D velocity field from the ECMWF reanalysis data.

We subdivide $X$ into a grid of $m=4116\times8=32928$ (longitude-latitude$\times$pressure) boxes, where all boxes have the same area in the longitude-latitude directions and a ``height'' of $(70-50)/8=20/8$ hPa in the pressure direction.
Following hydrostatic equilibrium considerations, we set the mass $p_i$ of box $B_i$ to be proportional to the base area of $B_i$ multiplied by the box ``height'' in hPa, and normalise so that $\sum_{i=1}^{32928}p_i=1$.
 We select $Q=250$ sample points in each grid box, uniformly distributed in the longitude-latitude direction and equally spaced in pressure direction.
 The $Q\times m$ images of these sample points are then covered by a grid of $n=51722$ boxes.

Repeating the approach of the two-dimensional study, the two largest singular values are computed to be $\sigma_1\approx1.0$ and $\sigma_2\approx0.9994$. A slice along the uppermost pressure level (50 hPa) of the optimal vectors $x$ and $y$ is shown in Figure~\ref{fig:3Dslice}.


\begin{figure}[hbt]
\centering
\subfigure[]{
\includegraphics[scale=.3]{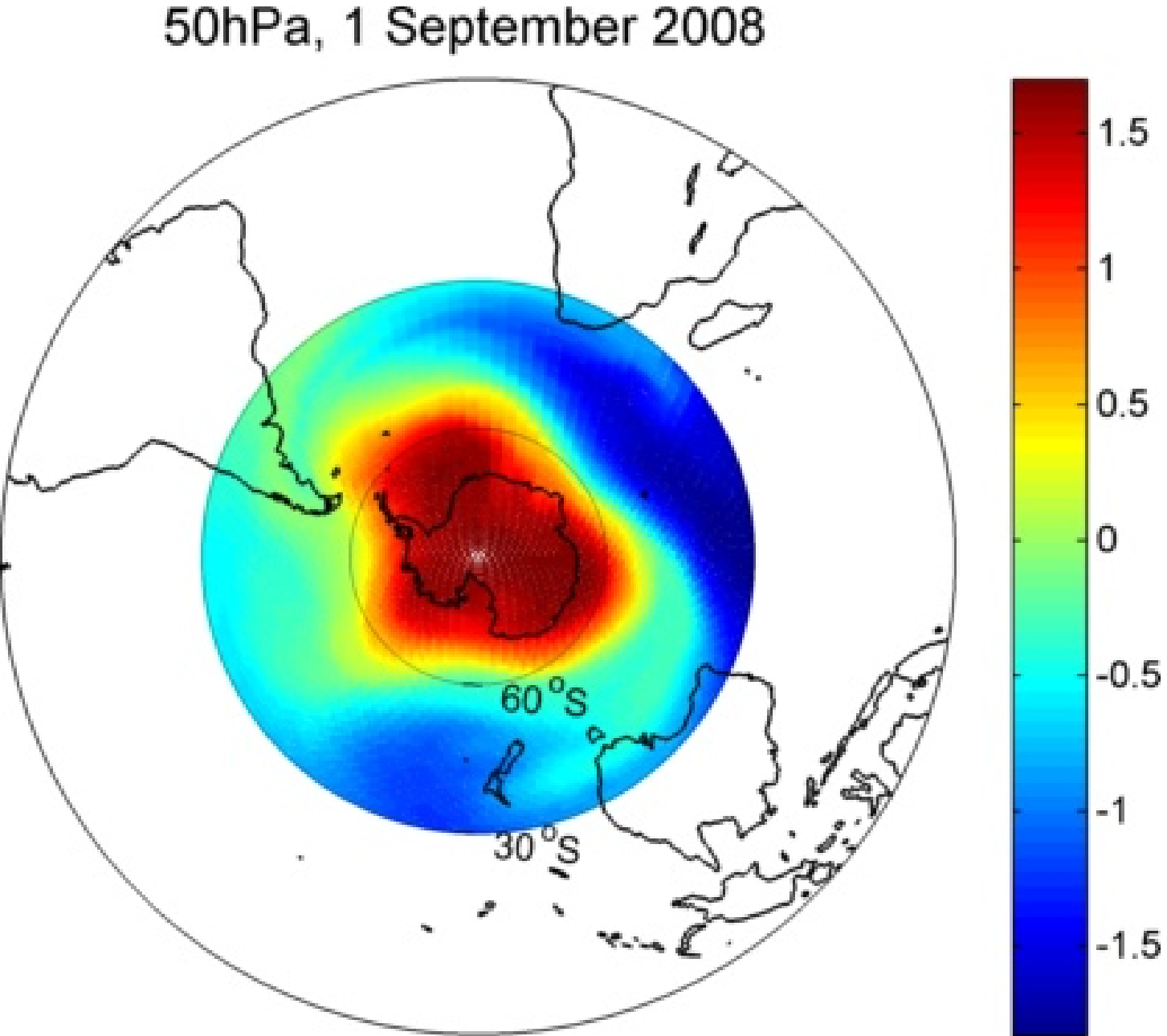}
}
\subfigure[]{
\includegraphics[scale=.3]{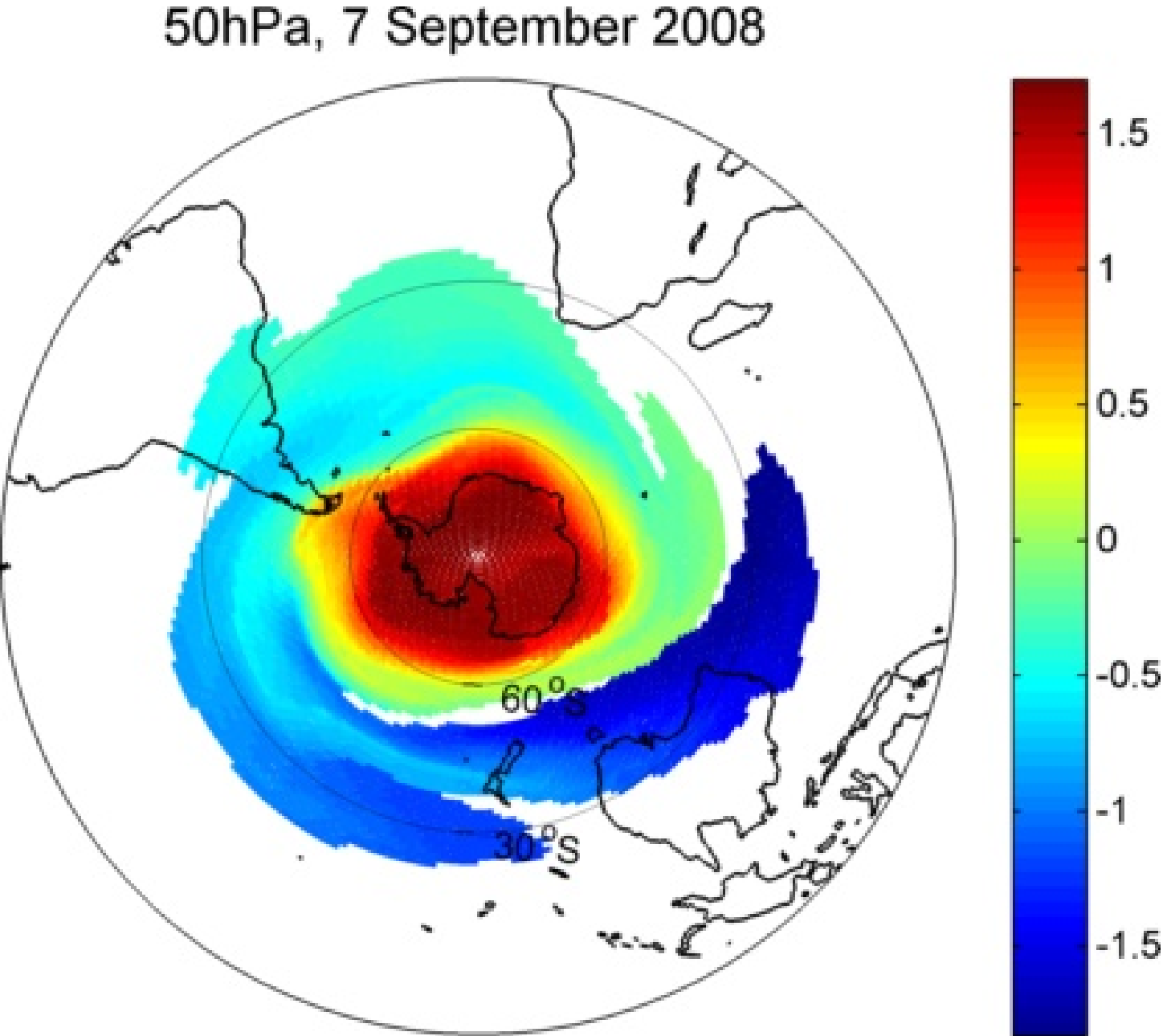}
}
\caption{\label{fig:3Dslice}\footnotesize{A slice of 3D vectors $x$ and $y$ along the 50hPa pressure level.}}
\end{figure}
Applying Algorithm \ref{alg1}, we compute the coherent sets $A_t$ and $A_{t+\tau}$ shown in Figure~\ref{fig:3dcoherents} with $\rho_{\mu}(A_{1},A_{14})\approx 0.9890$. Figures~\ref{fig:subfigz1} and \ref{fig:subfigz2} show that at 1 September 2008, a compact central domain with nearly vertical sides is extracted by Algorithm \ref{alg1}.  Figure \ref{fig:subfigz6} shows that after 6 days of flow, this set is advected both upwards and downwards, and that this advection is not uniform over all latitudes.  Figure \ref{fig:subfigz6} and Figure \ref{fig:subfigz8} (which gives a view from ``below''), demonstrate that the upward flow occurs primarily near the centre of the vortex (high latitudes), while the downward flow is concentrated around the periphery (lower latitudes). A bowl-like shape is evident in Figure \ref{fig:subfigz8} showing a thin layer at the core of the coherent set at 7 September, descending toward the troposphere near the edge of coherent set. This observation agrees with the motion of ozone masses in the lower stratosphere, where the mass in the mixing zone around the mid-latitude slowly moves downward and the mass in the vortex core moves within a thin stratospheric layer~\cite{Plumb02,Schoeberl_etal92}.



\begin{figure}[hbt]
\centering
\subfigure[]{
\includegraphics[scale=.225]{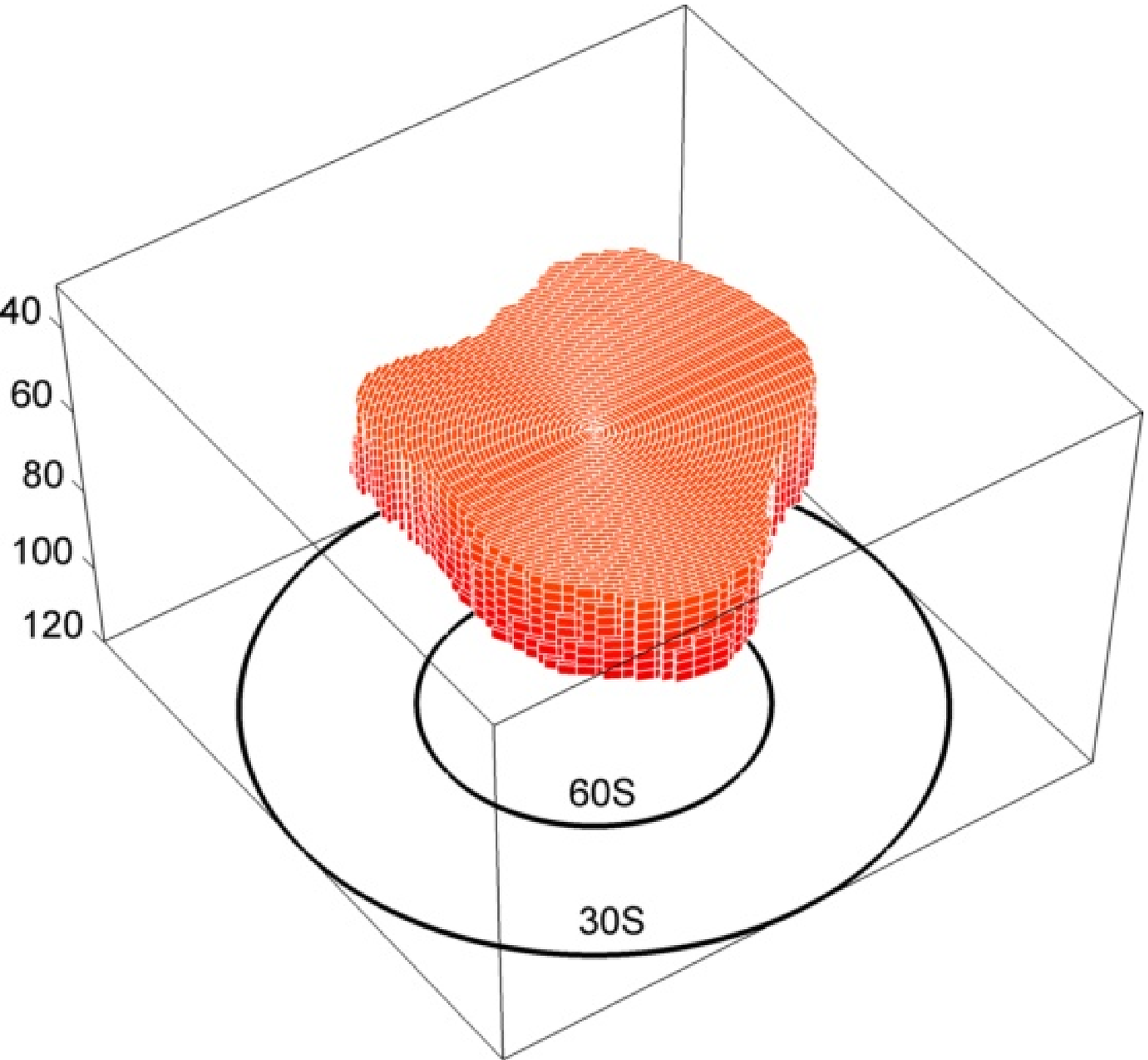}
\label{fig:subfigz1}
}
\subfigure[]{
\includegraphics[scale=.225]{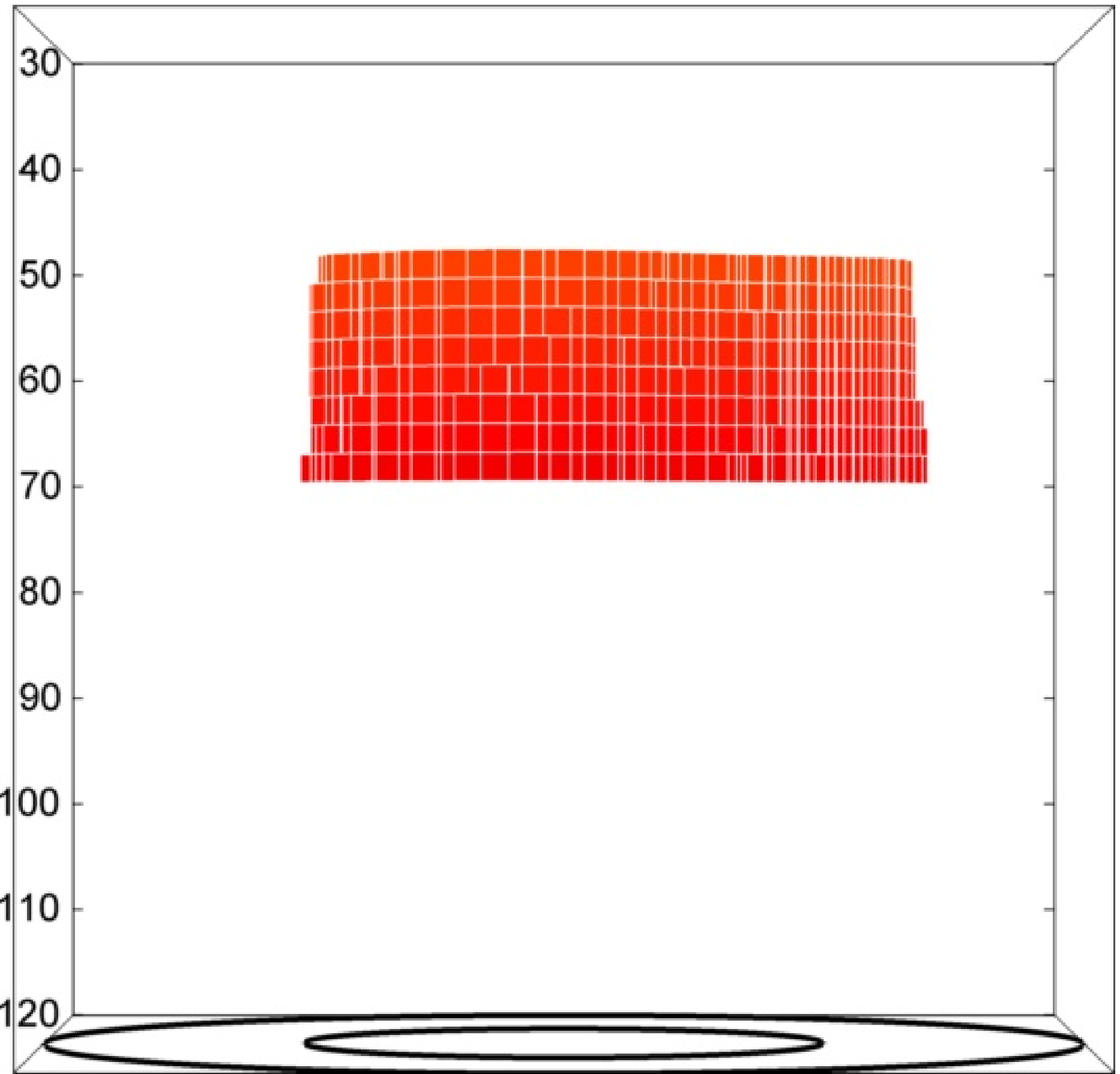}
\label{fig:subfigz2}
}
\subfigure[]{
\includegraphics[scale=.255]{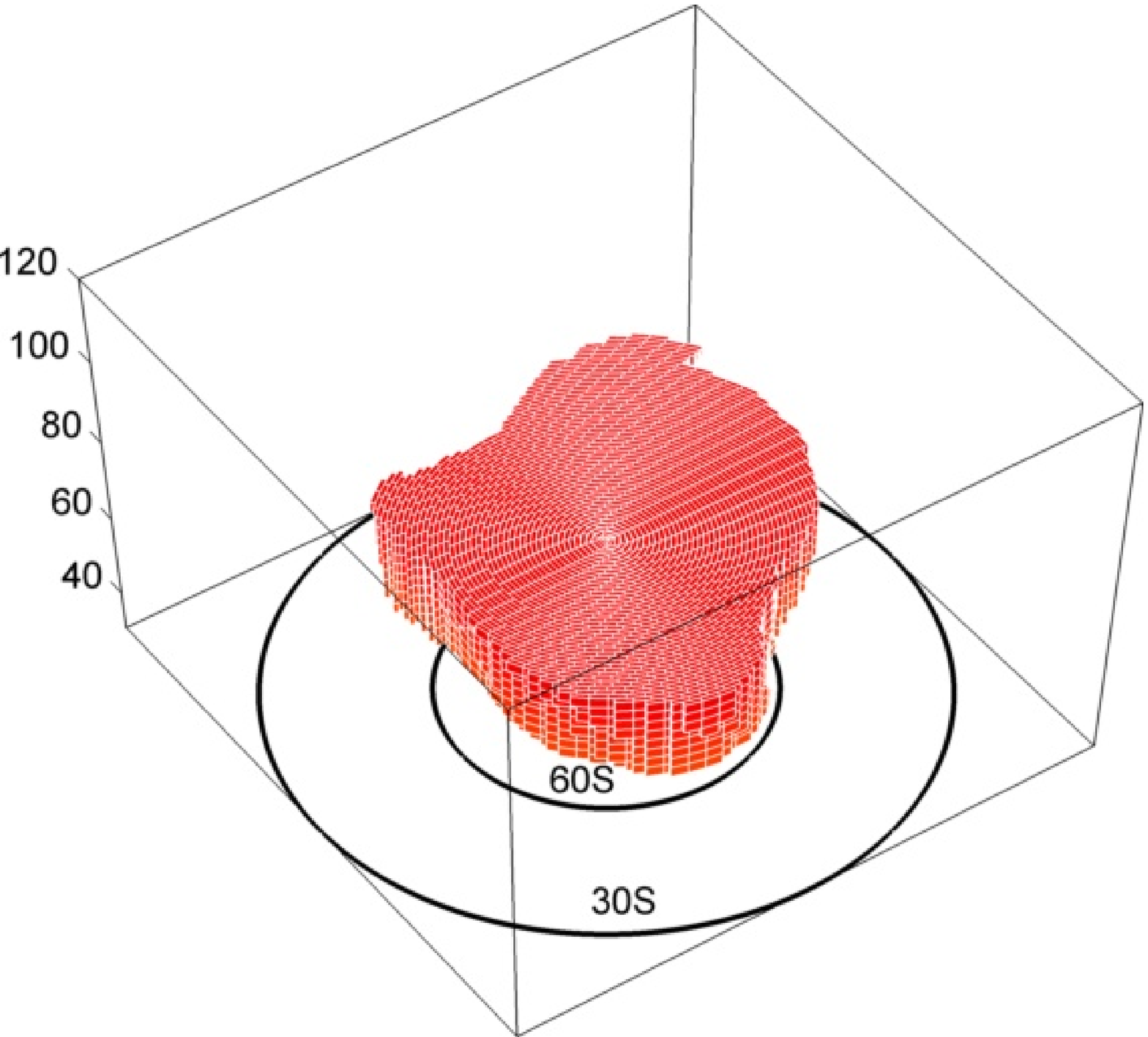}
\label{fig:subfigz4}
}
\subfigure[]{
\includegraphics[scale=.255]{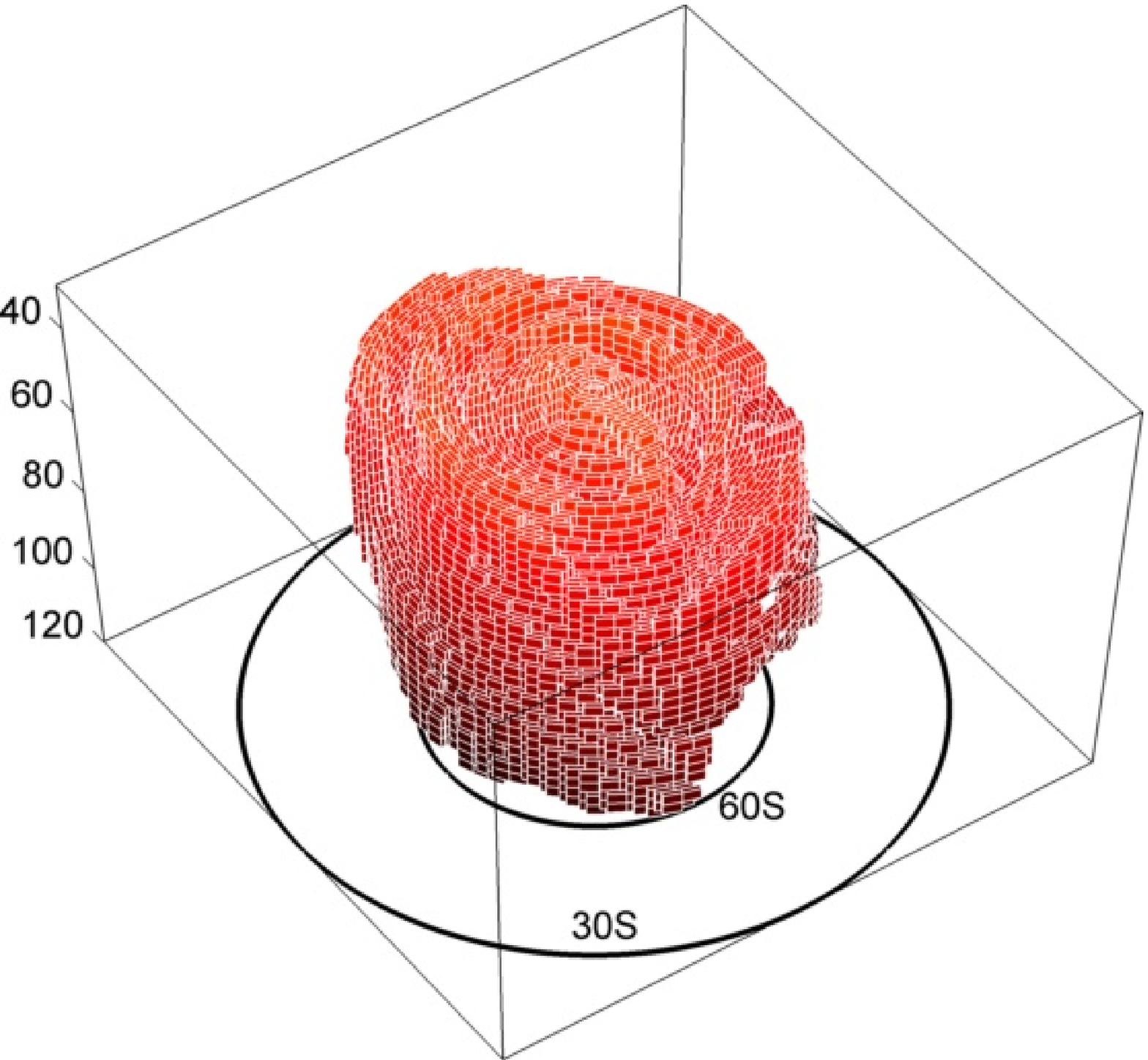}
\label{fig:subfigz5}
}
\subfigure[]{
\includegraphics[scale=.225]{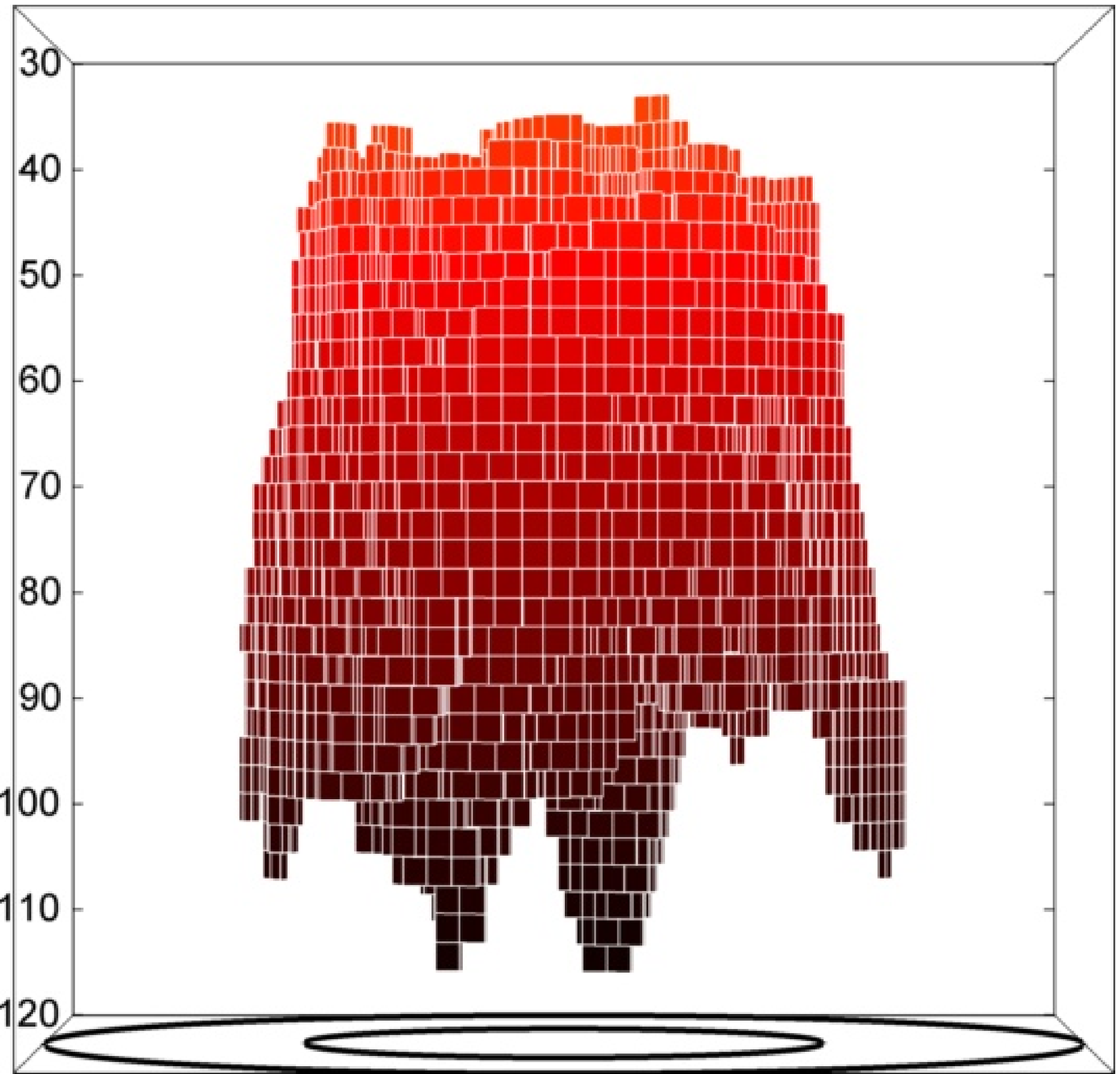}
\label{fig:subfigz6}
}
\subfigure[]{
\includegraphics[scale=.255]{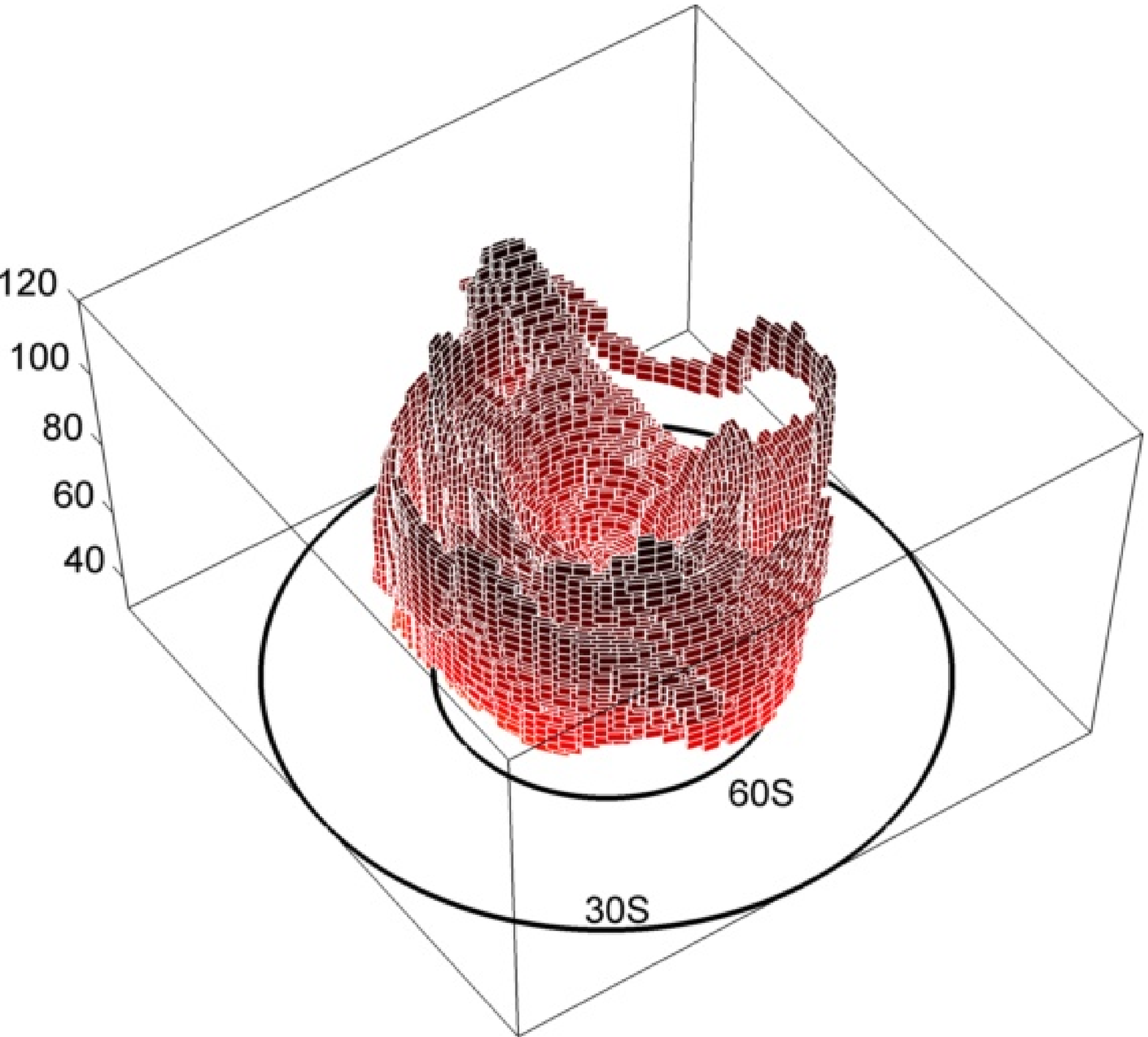}
\label{fig:subfigz8}
}\caption[]{\label{fig:3dcoherents}(a)-(c) show the optimal coherent sets at 1 September 2008 at different views. (d)-(e) show the optimal coherent sets at 7 September 2008 at different views.}
\end{figure}



\section{Conclusions}
We introduced a methodology for identifying minimally dispersive regions (coherent sets) in time-dependent flows over a finite period of time. Our approach directly used the time-dependent velocity fields to construct an ensemble description of the finite-time dynamics;  the Perron-Frobenius (or transfer operator).
The transport of mass is explicitly calculated in terms of a reference measure considered to be most appropriate for the application by the practitioner.
Singular vector computations of matrix approximations of the Perron-Frobenius operator directly yielded images of the coherent sets;  the left singular vector described the coherent region at the initial time and the right singular vector at the final time.
Our methodology is the first systematic transfer operator approach for handling time-dependent systems over finite time durations.
A particular feature of our approach is that one can focus on small subdomains of interest, rather than study the entire domain;  this leads to major computational savings.

In our first case study we used this new technique to show that an idealized stratospheric flow operates as two almost independent dynamical systems with a small amount of interaction across two Rossby wave regimes.
Our second case study utilised reanalysed velocity data sourced from the European Centre for Medium Range Weather Forecasting (ECMWF) to estimate the location of the Southern polar vortex.
Studying the dynamics on a two-dimensional isentropic surface, we found excellent agreement with traditional potential vorticity (PV) based approaches, and improved slightly over the PV methodology in terms of the coherence of the vortex.
We also used the full three-dimensional velocity field to determine the vortex location in three dimensions, a computation not easily carried out with standard applications of the PV approach.

\appendix
\section{Proof of Lemma \ref{complemma}}
We first show that the condition on $y$ in (\ref{relaxed}) is unnecessary.
\begin{equation}
\label{eqn1}\max_{x\in\mathbb{R}^m \atop y\in\mathbb{R}^n}\left\{\frac{\langle xL,y\rangle_q}{\|x\|_p\|y\|_q}: \langle x,\mathbf{1}\rangle_p=0\right\}=\max_{x\in\mathbb{R}^m \atop y\in\mathbb{R}^n}\left\{\frac{\langle xL,\frac{y}{\|y\|_q}\rangle_q}{\|x\|_p}: \langle x,\mathbf{1}\rangle_p=0\right\}=\max_{x\in\mathbb{R}^m}\left\{\frac{\|xL\|_q}{\|x\|_p}: \langle x,\mathbf{1}\rangle_p=0\right\},
\end{equation}
with the maximizing $y$ being $y=xL$.
Setting $y=xL$, we see that $\langle y,\mathbf{1}\rangle_q=\langle xL,\mathbf{1}\rangle_q=\langle x,\mathbf{1}L^*\rangle_p=\langle x,\mathbf{1}\rangle_p=0$ (it is straightforward to check $\mathbf{1}L^*=\textbf{1}$;  $L^*=P^\top$).
Thus, since the maximizing $y$ in (\ref{eqn1})  satisfies $\langle y,\mathbf{1}\rangle_q=0$ when $\langle x,\mathbf{1}\rangle_p=0$, we see that the value of (\ref{relaxed}) equals the value of the LHS of (\ref{eqn1}), and both (\ref{relaxed}) and (\ref{eqn1}) have the same maximizing $x$ and $y$.

We now convert the RHS of (\ref{eqn1}) to a maximization in the standard $\ell_2$ norm by noting that $\langle x_1,x_2\rangle_p=\langle x_1\Pi_p^{1/2},x_2\Pi_p^{1/2}\rangle_2$ and $\langle y_1,y_2\rangle_q=\langle y_1\Pi_q^{1/2},y_2\Pi_q^{1/2}\rangle_2$.
\begin{equation}
\mbox{RHS of (\ref{eqn1})}=\max_{x\in\mathbb{R}^m}\left\{\frac{\|xL\Pi_q^{1/2}\|_2}{\|x\Pi_p^{1/2}\|_2}: \langle x\Pi^{1/2}_p,\mathbf{1}\Pi^{1/2}_p\rangle_2=0\right\}=\max_{\hat{x}\in\mathbb{R}^m}\left\{\frac{\|\hat{x}\Pi^{-1/2}_pL\Pi_q^{1/2}\|_2}{\|\hat{x}\|_2}: \langle \hat{x},p^{1/2}\rangle_2=0\right\},
\end{equation}
where we have made the substitution $\hat{x}=x\Pi_p^{1/2}$. We claim that the leading singular value of $\Pi^{-1/2}_pL\Pi_q^{1/2}=\Pi^{1/2}_pP\Pi_q^{-1/2}$ is 1, with corresponding left singular vector $p^{1/2}$.

To prove this claim, we show that 1 is the leading singular value of $L$ with corresponding left singular vector $\mathbf{1}$ (where $L$ is always considered as a linear mapping from $\langle \cdot,\cdot\rangle_p$ to $\langle \cdot,\cdot\rangle_q$).
Since $\mathbf{1}L=\mathbf{1}$ and $\mathbf{1}L^*=\mathbf{1}$, one has $\mathbf{1}LL^*=\mathbf{1}$;  also $LL^*$ is irreducible iff $PP^\top$ is irreducible.  By the Perron-Frobenius Theorem (eg.\ Thm 1.4 and 2.1 \cite{berman_plemmons}), 1 is the largest real eigenvalue of $LL^*$, and is simple; hence the largest singular value of $L$ is 1 and the left and right singular vectors are $\mathbf{1}\in\mathbb{R}^m$ and $\mathbf{1}\in\mathbb{R}^n$ respectively.

The result now follows from the Courant-Fischer theorem for symmetric matrices (see eg. Thm.\ 4.2.11 \cite{horn_johnson}), standard properties of singular vectors and the computation $y=xL=\hat{x}\Pi_p^{1/2}L=(\hat{x}\Pi_p^{-1/2}L\Pi_q^{1/2})\Pi_q^{-1/2}=\hat{y}\Pi_q^{-1/2}$ where $\hat{y}$ is the right singular vector of $\Pi_p^{-1/2}L\Pi_q^{1/2}$ corresponding to $\sigma_2$.

\bibliographystyle{unsrt}
\bibliography{Bibchaos2010}

\end{document}